 \documentclass[final,1p,times]{elsarticle}

\usepackage{amssymb}
 \usepackage{amsthm}
\usepackage{amsmath,amssymb,amsopn,amsfonts,mathrsfs,amsbsy,amscd}
\usepackage{longtable}

\journal{Journal of Differential Geometry and its applications}

\newcommand{\kr}{\mathfrak{k}}
\newcommand{\V}{{\mathcal V}}

\newcommand{\prs}{\langle\;,\;\rangle}
\newcommand{\prsm}{\langle\;,\;\rangle_{TM}}

\newcommand{\too}{\longrightarrow}

\newcommand{\esp}{\quad\mbox{and}\quad}
\newcommand{\wi}{\widetilde}

\newcommand{\g}{\mathfrak{g}}

\newcommand{\ad}{{\mathrm{ad}}}
\newcommand{\Ad}{{\mathrm{Ad}}}
\newcommand{\tr}{{\mathrm{tr}}}
\newcommand{\ric}{{\mathrm{ric}}}

\newcommand{\p}{{\mathfrak{p} }}

\newcommand{\na}{\nabla}

\newcommand{\al}{\alpha}
\newcommand{\be}{\beta}
\newcommand{\ga}{\gamma}
\newcommand{\Ga}{\Gamma}
\newcommand{\e}{\epsilon}

\newcommand{\la}{\lambda}

\newcommand{\de}{\delta}

\newtheorem{theo}{Theorem}[section]
\newtheorem{pr}{Proposition}[section]

\newtheorem{co}{Corollary}[section]

\newtheorem{exem}{Example}
\newtheorem{remark}{Remark}

\font\bb=msbm10

\def\R{\hbox{\bb R}}

\def\Co{\hbox{\bb C}}

\begin{document}

\begin{frontmatter}


 

\title{ The geometry of  the Sasaki metric on the sphere bundle of  Euclidean Atiyah vector bundles  }

 \author[label1]{Mohamed Boucetta}
 \address[label1]{Universit\'e Cadi-Ayyad\\
 	Facult\'e des sciences et techniques\\
 	BP 549 Marrakech Maroc\\e-mail: m.boucetta@uca.ac.ma}
 \author[label2]{Hasna Essoufi }
 \address[label2]{Universit\'e Cadi-Ayyad\\
 	Facult\'e des sciences et techniques\\
 	BP 549 Marrakech Maroc\\e-mail: essoufi.hasna@gmail.com}
 
 



\begin{abstract} Let $(M,\prsm)$ be a Riemannian manifold. It is well-known that the Sasaki metric on $TM$ is very rigid but it has nice properties when restricted to $T^{(r)}M=\{u\in TM,|u|=r \}$. In this paper, we consider a general situation where we replace $TM$ by a  vector bundle $E\too M$ endowed with a Euclidean product $\prs_E$ and a connection $\na^E$ which preserves $\prs_E$. We define the Sasaki metric on $E$ and we consider its restriction $h$ to $E^{(r)}=\{a\in E,\langle a,a\rangle_E=r^2  \}$. We study the Riemannian geometry of $(E^{(r)},h)$ generalizing many results first obtained on $T^{(r)}M$ and establishing new ones.
	We apply the results obtained in this general setting to the class of Euclidean Atiyah vector bundles introduced by the authors in \cite{boucettaessoufi}. Finally, we prove that any unimodular three dimensional Lie group $G$ carries a left invariant Riemannian metric such that $(T^{(1)}G,h)$ has a positive scalar curvature. 
\end{abstract}

\end{frontmatter}








\section{Introduction}\label{section1}

Through this paper, a Euclidean vector bundle is a vector bundle $\pi_E:E\too M$ endowed with  $\prs_E\in\Ga(E^*\otimes E^*)$ which is bilinear symmetric and definite positive in the restriction to each fiber.

Let $(M,\prsm)$ be a Riemannian manifold of dimension $n$, $\pi_E:E\too M$ a vector bundle of rank $m$ endowed with a Euclidean product $\prs_E$ and  a linear connection $\na ^E$  which preserves $\prs_E$. Denote by $K:TE\too E$ the connection map of $\na^E$ locally given   by
\[ K\left( \sum_{i=1}^n b_i\partial_{x_i}+\sum_{j=1}^mZ_j\partial_{\mu_j}\right)=
\sum_{l=1}^m\left( Z_l+\sum_{i=1}^n\sum_{j=1}^m b_i\mu_j\Ga_{ij}^l\right)s_l, \]where $(x_1,\ldots,x_n)$ is a system of local coordinates, $(s_1,\ldots,s_m)$ is a basis of local sections of $E$,  $(x_i,\mu_j)$ the associated system of coordinates on $E$ and $\na_{\partial_{x_i}}^E s_j=\sum_{l=1}^m\Ga_{ij}^l s_l$. Then
\[ TE=\ker d\pi_E\oplus \ker K. \]
The Sasaki metric $g_s$ on $E$ is the Riemannian metric given by
\[ g_s(A,B)=\langle d\pi_E(A),d\pi_E(B) \rangle_{TM}+   \langle K(A),K(B) \rangle_{E},\quad A,B\in T_aE.  \]For any $r>0$, the sphere bundle of radius $r$ is the hypersurface $E^{(r)}=\left\{a\in E,\langle a,a\rangle_E=r^2      \right\}$. 

They are  two classes  of such Euclidean vector bundles naturally associated to a Riemannian manifold. 

We refer to the first one as the classical case. It is the case where $E=TM$, $\prs_E=\prsm$ and $\na^E$ is the Levi-Civita connection of $(M,\prsm)$. 

The second case will be called the {\it Euclidean Atiyah vector bundle} associated  to a Riemannian manifold. It has been introduced by the authors in \cite{boucettaessoufi}. It is defined as follows. 

Let $(M,\prsm)$ be a Riemannian manifold,   $\mathrm{so}(TM)=\bigcup_{x\in M}\mathrm{so}(T_xM)$ where $\mathrm{so}(T_xM)$ is the vector space of skew-symmetric endomorphisms of $T_xM$ and $k>0$. The Levi-Civita connection $\na^M$ of $(M,\prsm)$ defines a connection on the vector bundle $\mathrm{so}(TM)$ which we will denote in the same way and it is given, for any $X\in\Ga(TM)$ and $F\in\Ga(\mathrm{so}(TM))$, by
\[ \na^M_XF(Y)=\na^M_X(F(Y))-F(\na_X^MY). \]
The {\it Atiyah Euclidean vector bundle}\footnote{The origin of this vector bundle and the justification of its name can found in \cite{boucettaessoufi}.} associated to $(M,\prsm,k)$ is the triple $(E(M,k),\prs_k,\na^E)$ where
$E(M,k)=TM\oplus \mathrm{so}(TM)\too M$, $\prs_k$ and $\na^E$ are a Euclidean product  and a connection on $E(M,k)$  given, for any $X,Y \in \Ga(TM)$ and $F,G \in \Ga(\mathrm{so}(TM)),$ by
\begin{eqnarray*}
	\na^E_X Y&=& \na^M_X Y + H_X Y,\; \na^E_X F=   H_X F+\na^M_X F,\\
	\langle X+F,Y+G \rangle_k &=& \langle X,Y \rangle_{TM} -k\; \mathrm{tr}(F\circ G),
\end{eqnarray*}
where $R^M$ is the curvature tensor of $\na^M$ given by $R^M(X,Y)=\na^M_{[X,Y]}-\left( \na_X^M\na_Y^M-\na_Y^M\na_X^M\right),$
 \begin{equation}\label{H0} H_XY=-\frac12 R^M(X,Y)\;\esp\; \langle H_XF,Y\rangle_{TM}=-\frac12 k\; \mathrm{tr}(F\circ R^M(X,Y)). \end{equation}
The connection $\na^E$ preserves $\prs_k$ and its curvature $R^{\na^E}$ plays a key role in the study of $(E^{(r)}(M,k)$ endowed with the Sasaki metric. Since $R^{\na^E}$ depends only on $(M,\prsm,k)$, we will call it the {\it supra-curvature} of $(M,\prsm,k)$.

This paper has two goals:
\begin{enumerate}
	\item The study of the Riemannian geometry of $E^{(r)}$ endowed with the Riemannian metric $h$  restriction of $g_s$ in order to generalize all the results obtained in the classical case. We refer to  \cite{bo, kowalski} for a survey on the geometry of $(T^{(r)}M,h)$.
	\item The application of the results obtained in the general case to the Euclidean Atiyah vector bundle $E^{(r)}(M,k)$ endowed with the Sasaki metric. We will show that the geometry of $(E^{(r)}(M,k),h)$ is so rich and by doing so we open new horizons for further explorations.
\end{enumerate}Let us give now the organization of this paper. In Section \ref{section2}, we give the different curvatures of $(E^{(r)},h)$. In Section \ref{section3} we derive  sufficient conditions  for which $(E^{(r)},h)$ has either nonnegative sectional curvature, positive Ricci curvature,  positive or constant scalar curvature. In Section \ref{section4},  we first compute the supra-curvature of different classes of Riemannian manifolds and we characterize those with vanishing supra-curvature (see Theorem \ref{supra}). Then we  perform a detailed study of $(E^{(r)}(M,k),h)$ having in mind the results obtained in Section \ref{section3}.
In Section \ref{section5}, we prove that any unimodular three dimensional Lie group $G$ carries a left invariant Riemannian metric such that $(T^{(1)}G,h)$ has a positive scalar curvature.

\section{Sectional curvature, Ricci curvature and scalar curvature of the Sasaki metric on  sphere bundles }	\label{section2}
Through this section, $(M,\prsm)$ is a $n$-dimensional Riemannian manifold and $\pi_E:E\too M$ a vector bundle  of rank $m$ endowed with a Euclidean product $\prs_E$ and a linear connection $\na^E$  for which $\prs_E$ is parallel. We shall denote by $\na^M$ the Levi-Civita connection of $(M,\prsm)$,  by $R^M$ and $R^{\na^E}$ the tensor curvatures of $\na^M$ and $\na^E$, respectively. We use the convention
\[ R^M(X,Y)=\na^M_{[X,Y]}-\left(\na_X^M\na_Y^M-\na_Y^M\na_X^M \right)\esp
R^{\na^E}(X,Y)=\na^E_{[X,Y]}-\left(\na^E_X\na^E_Y-\na^E_Y\na^E_X\right).  \]
The derivative of $R^{\na^E}$ with respect to $\na^M$ and $\na^E$ is the tensor field $\na_{X}^{M,E}(R^{\na^E})$ given,  for any $X,Y,Z \in \Ga(TM)$, $\al\in \Ga(E)$, by \begin{equation}\label{eqrme}\na_X^{M,E}(R^{\na^E})(Y,Z,\al)=\na_X^{E}(R^{\na^E}(Y,Z)\al)-R^{\na^E}(\na_X^{M}Y,Z)\al-R^{\na^E}(Y,\na_X^{M}Z)\al-R^{\na^E}(Y,Z)\na_X^E\al.\end{equation}
Let $K^M$, $\ric^M$ and $s^M$ denotes the sectional curvature, the Ricci curvature and the scalar curvature of $(M,\prsm)$, respectively. An element of  $E$ will be denoted by $(x,a)$ with $x\in M$ and $a\in E_x$.

We recall the definition of the Sasaki metric $g_S$ on $E$,  we consider its restriction $h$ to the sphere bundles $E^{(r)}=\left\{a\in E,\langle a,a\rangle_E=r^2      \right\}$ $(r>0)$ and we give the expressions of the different curvatures of $(E^{(r)},h)$. 

 For any $(x,a)\in E$  there exists an injective linear map $h^{(x,a)}:T_xM\too T_{(x,a)}E$ given in a coordinates system $(x_i,\be_j)$ on $E$ associated to a coordinates $(x_i)_{i=1}^n$ on $M$ and a local trivialization $(s_1,\ldots,s_m)$ of $E$  by
\begin{equation*}\label{he} h^{(x,a)}(u)=\sum_{i=1}^nu_i\partial_{x_i}-\sum_{k=1}^m\left(\sum_{i=1}^n\sum_{j=1}^m u_i\be_j\Ga_{ij}^k\right)\partial_{\be_k},\end{equation*}where
$$u=\sum_{i=1}^nu_i\partial_{x_i},\; \na^E_{\partial_{x_i}}s_j=\sum_{k=1}^m\Ga_{ij}^ks_k\esp a=\sum_{i=1}^m\be_is_i.$$ Moreover,  if $ \mathcal{H}_{(x,a)}E$ denotes the image of $h^{(x,a)}$ then   
\begin{equation*}\label{eq21e} TE=\V E\oplus \mathcal{H} E,\end{equation*}where $\V E=\ker d\pi_E$.
For any $\al\in\Ga(E)$ and for any $X\in\Ga(TM)$, we denote by $\al^v\in\Ga(TE)$ and $X^h\in\Ga(TE)$ the  vertical and horizontal vector field associated to $\al$ and $X$. The flow of $\al^v$ is given by $\Phi^\al(t,(x,a))=(x,a+t\al(x))$ and $X^h$ is given by $X^h(x,a)=h^{(x,a)}(X(x))$. \\

The Sasaki metric $g_s$ on $E$  is determined by the formulas
\begin{eqnarray*}
	g_s(X^h,Y^h)&=&\langle X,Y \rangle_{TM} \circ \pi_E,\;  g_s(\al^v,\be^v)=\langle \al,\be \rangle_E \circ \pi_E\esp   g_s(X^h,\al^v)=0,
\end{eqnarray*}
for all $X,Y \in \Ga(TM)$ and $\al,\be \in \Ga(E)$. 

For any $X\in\Ga(TM)$ and $\al \in \Ga(E)$, $X^h$ is tangent to $E^{(r)}$ however $\al^v$ is not tangent to $E^{(r)}$. So we define 
the {\it tangential lift} of $\al$  by
\[ \al^t(x,a)=\al^v(x,a)-\langle \al,a\rangle_E \frac{U(x,a)}{r^2},\quad (x,a)\in E, \]
where $U$ is  the vertical vector field on $E$ whose flow is given by $\Phi(t,(x,a))=(x,e^ta)$.
We have\[ T_{(x,a)}E^{(r)}=\left\{ X^h+\al^t\;/\; X\in T_xM\;\text{and}\; \al\in E_x \;\text{with}\;\langle \al,a \rangle_E=0 \right\}. \]
The restriction $h$ of $g_S$ to $E^{(r)}$ is given by 
\begin{eqnarray*}\label{indsasaki}
h(X^h,Y^h)&=& \langle X,Y\rangle_{TM}\circ\pi_E  ,\nonumber\quad h(X^h,\al^t)=0,\\ 
h(\al^t,\be^t)(x,a)&=& \langle\al,\be\rangle_E-\frac{\langle \al,a\rangle_E\langle \be,a\rangle_E}{r^2}=\langle\bar{\al},\bar{\be}\rangle_E,
\end{eqnarray*}
where $\al,\be\in\Ga(E),\; X,Y\in\Ga(TM)$ and $ \bar{\al}=\al-\frac{\langle \al,a\rangle_E}{r^2}a$.

The following proposition can be established in the same way as the classical case where $E=TM$, $\prs_E=\prsm$ and $\na^E=\na^M$.
\begin{pr}	\label{prb}
	We have
	\[ [\al^t,\be^t]=\frac{\langle \al,a\rangle_E}{r^2}\be^t-\frac{\langle \be,a\rangle_E}{r^2}\al^t,\;[X^h,\al^t]=\left(\na^E_{X}\al  \right)^t
	\esp [X^h,Y^h](x,a)=[X,Y]^h(x,a)+(R^{\na^E}(X,Y)a)^t, \]
	where $R^{\na^E}$ is the curvature of ${\na^E}$ given by $R^{\na^E}(X,Y)=\na^E_{[X,Y]}-\left(\na^E_X\na^E_Y-\na^E_Y\na^E_X\right)$.
\end{pr}

To compute the Riemannian invariants of $(E^{(r)},h)$ (Levi-Civita connection and the different curvatures), we will use the following facts:
\begin{enumerate}
	\item[$(i)$] The projection $\pi_E:(E^{(r)},h)\too (M,\prsm)$ is a Riemannian submersion with totally geodesic fibers and hence the different Riemannian invariants can be computed by using O'Neill formulas (see \cite[chap. 9]{bes}). Here the  O'Neill shape tensor, say $B$, is given by the expression of $[X^h,Y^h]$. So, by virtue of Proposition \ref{prb}, we get
	\begin{equation}\label{b}
	B_{X^h}Y^h((x,a))=\frac12\V[X^h,Y^h](x,a)=\frac12(R^{\na^E}(X,Y)a)^v
	=\frac12(R^{\na^E}(X,Y)a)^t,
	\end{equation} $B_{\al^t}=0$ and $h(B_{X^h}\al^t,Y^h)=-h(B_{X^h}Y^h,\al^t)$ for any $\al\in\Ga(E)$,  $X,Y\in\Ga(TM)$ and $(x,a)\in E$.
	\item[$(ii)$] O'Neill's formulas involve the Riemannian invariants of $(M,\prsm)$, the tensor $B$ and the Riemannian invariants of the restriction of $h$ to the fibers.\end{enumerate}

Based on these facts, 
the Levi-Civita connection $\overline{\na}$ of $(E^{(r)},h)$ is given by
\begin{eqnarray}
\overline{\na}_{X^h}Y^h(x,a)&=&(\na_X^MY)^h(x,a)+\frac12(R^{\na^E}(X,Y)a)^t,\;\overline{\na}_{X^h}\al^t=B_{X^h}\al^t+(\na^E_{X}\al)^t,\;
\overline{\na}_{\al^t}X^h=B_{X^h}\al^t,\nonumber\\
(\overline{\na}_{\al^t}\be^t)(x,a)&=&-\frac{\langle \be,a\rangle}{r^2}\al^t\esp 
h(B_{X^h}\al^t,Y^h)=-h(B_{X^h}Y^h,\al^t),\label{levih}
\end{eqnarray} $X,Y\in\Ga(TM)$, $\al,\be\in\Ga(E)$ and $(x,a)\in E$.
Note that if $(X_i)_{i=1}^n$ is a local orthonormal frame of $TM$, $X\in\Ga(TM)$ and $\al\in\Ga(E)$ 
\begin{equation}\label{bxalpha} B_{X^h}\al^t=\frac12\sum_{i=1}^n\langle R^{\na^E}(X,X_i)\al,a\rangle_E X_i^h. \end{equation}

\begin{remark} When $E=TM$, $\prs_E=\prsm$ and $\na^E=\na^M$ we have a simple expression of $B_{X^h}\al^t$ thanks to the symmetries of $R^{\na^E}=R^M$, namely,
	\begin{equation}\label{tm} (B_{X^h}Y^t)(x,a)=\frac12R^M(Y(x),a)X(x),\;\quad X,Y\in\Ga(TM). \end{equation}
	
\end{remark}

A direct computation shows that the tensor curvature, the Ricci curvature and the scalar curvature of the fibers are given by 
\begin{eqnarray*}
	R^v(\al^t,\be^t)\ga^t &=&\frac{1}{r^2}\left( h(\al^t,\ga^t)\be^t-h(\be^t,\ga^t)\al^t\right),\;
	\ric^v(\al^t,\be^t)=\frac1{r^2}(m-2)h(\al^t,\be^t)\esp
	s^v =\frac1{r^2}(m-1)(m-2).
\end{eqnarray*}

 In order to  compute  the different curvatures of $(E^{(r)},h)$, we need the following formulas.
 \begin{pr}\label{f} For any $X,Y,Z\in\Ga(TM)$, $\al,\be\in\Ga(E)$ and $(x,a)\in E$, we have
 	\[ h((\overline{\na}_{X^h}B)_{Y^h}Z^h,\al^t)(x,a)=-
 	\frac12\langle \na_X^{M,E}(R^{\na^E})(Y,Z,\al),a\rangle_E. \]Moreover, if
 	$\langle\al(x),a\rangle_E=\langle\be(x),a\rangle_E=0$ then
 	\begin{eqnarray*}
 		h((\overline{\na}_{\al^t}B)_{X^h}Y^h,\be^t)(x,a)&=&\frac12\langle R^{\na^E}(X,Y)\al,\be\rangle_E(x)+ h(B_{Y^h}\al^t,B_{X^h}\be^t)(x,a)-h(B_{X^h}\al^t,B_{Y^h}\be^t)(x,a).	
 	\end{eqnarray*}
 	\end{pr}
 	
 	\begin{proof} Suppose first that $\langle\al(x),a\rangle_E=\langle\be(x),a\rangle_E=0$. We have
 		\begin{eqnarray*}
 			h((\overline{\na}_{\al^t}B)_{X^h}Y^h,\be^t)&=&
 			h(\overline{\na}_{\al^t}(B_{X_h}Y^h),\be^t)-
 			h(B_{\overline{\na}_{\al^t}X^h}Y^h,\be^t)-
 			h(B_{X^h}\overline{\na}_{\al^t}Y^h,\be^t)\\
 			&=&\al^t.h(B_{X_h}Y^h,\be^t)-h(B_{X_h}Y^h,\overline{\na}_{\al^t}\be^t)+
 		h(B_{Y^h}\overline{\na}_{\al^t}X^h,\be^t)+h(\overline{\na}_{\al^t}Y^h,B_{X^h}\be^t)\\
 		&=&	\al^t.h(B_{X_h}Y^h,\be^t)-h(B_{X_h}Y^h,\overline{\na}_{\al^t}\be^t)
 		+h(B_{Y^h}\al^t,B_{X^h}\be^t)-h(B_{X^h}\al^t,B_{Y^h}\be^t).
 		\end{eqnarray*}From \eqref{levih} and the definition of $\al^t$ we get
 		\[ \overline{\na}_{\al^t}\be^t(x,a)=0\esp (\al^t.h(B_{X_h}Y^h,\be^t))(x,a)
 		=(\al^v.h(B_{X_h}Y^h,\be^t))(x,a).  \]
 		But
 		\begin{eqnarray*}
 			\al^v.h(B_{X_h}Y^h,\be^t)(x,a)&=&\frac{d}{dt}_{|t=0}h(B_{X_h}Y^h(a+t\al),\be^t(a+t\al))\\
 			&=&\frac{d}{dt}_{|t=0}\left[h(B_{X_h}Y^h(a+t\al),\be^v(a+t\al))
 			-\frac1{r^2}\langle \be,a+t\al\rangle_Eh(B_{X_h}Y^h(a+t\al),U(a+t\al))\right]\\
 			&=&\frac12\frac{d}{dt}_{|t=0}\langle R^{\na^E}(X,Y)(a+t\al),\be\rangle_E(x)\\
 			&=&\frac12\langle R^{\na^E}(X,Y)\al,\be\rangle_E(x),
 		\end{eqnarray*}which complete to establish the second formula.

 	On the other hand,
 	\begin{eqnarray*}
 		h((\overline{\na}_{X^h}B)_{Y^h}Z^h,\al^t)(x,a)&=&h(\overline{\na}_{X^h}(B_{Y^h}Z^h),\al^t)(x,a)-
 		h(B_{\overline{\na}_{X^h}Y^h}Z^h,\al^t)(x,a)-h(B_{Y^h}\overline{\na}_{X^h}Z^h,\al^t)(x,a)\\
 		&=&X^h.h(B_{Y^h}Z^h,\al^t)(x,a)-\frac12\langle R^{\na^E}(Y,Z)a,\na^E_X\al\rangle_E  -  \frac12  \langle R^{\na^E}(\na^M_XY,Z)a,\al\rangle_E \\&&-\frac12 \langle R^{\na^E}(Y,\na^M_XZ)a,\al\rangle_E\\
 		&=&\frac12\langle R^{\na^E}(Y,Z)\na^E_X\al+
 		R^{\na^E}(\na^M_XY,Z)\al+R^{\na^E}(Y,\na^M_XZ)\al,a \rangle_E+
 		X^h.h(B_{Y^h}Z^h,\al^t)(x,a).
 	\end{eqnarray*}	The key point is that if $\phi_t^X(x)$ is the integral curve of $X$ passing through $x$ then the integral curve of $X^h$ at $a$ is the $\na^E$-parallel section $a^t$ along $\phi_t^X(x)$ with $a^0=a$. So
 	\begin{eqnarray*}
 	X^h.h(B_{Y^h}Z^h,\al^t)(x,a)&=&\frac{d}{dt}_{|t=0} h(B_{Y^h}Z^h,\al^t)(a^t)\\
 	&=&-\frac12\frac{d}{dt}_{|t=0}\langle R^{\na^E}(Y(\phi_t^X(x)),Z(\phi_t^X(x)))\al(\phi_t^X(x)),a^t\rangle_E\\
 	&=&-\frac12\langle \na_{X}^E(R^{\na^E}(Y,Z)\al)(x),a\rangle_E.
 	\end{eqnarray*}
 		This completes the proof.
 	\end{proof}

\begin{pr}\label{ke}
	 Let $P\subset T_{(x,a)}E^{(r)}$ be a plane. Then: 
		\begin{enumerate}
			\item If $\mathrm{rank}(E)=2$ then there exists a basis $\{ X^h+\al^t,Y^h  \}$ of $P$ satisfying
			\[\al\in E_x,\;X,Y\in T_xM,\; |X|^2+|\al|^2=|Y|^2=1,\;\langle X,Y\rangle_{TM}=0\esp \langle\al,a\rangle_E=0. \]
			The sectional curvature of $(E^{(r)},h)$ at $P$ is given by
			\begin{eqnarray*}
				K(P)
				&=&\langle R^M(X,Y)X,Y\rangle_{TM}-\frac34| R^{\na^E}(X,Y)a|^2+ \frac14\sum_{i=1}^n\langle R^{\na^E}(Y,X_i)\al,a\rangle_E^2  \\&&+\langle \na_Y^{M,E}(R^{\na^E})(X,Y,\al),a\rangle_E.
			\end{eqnarray*}
			\item If $\mathrm{rank}(E) \geq 3$ then there exists a basis $\{ X^h+\al^t,Y^h+\be^t  \}$ of $P$ satisfying
			\[\al,\be\in E_x,\;X,Y\in T_xM,\; |X|^2+|\al|^2=|Y|^2+|\be|^2=1,\;\langle X,Y\rangle_{TM}=\langle\al,\be\rangle_E=0\esp \langle\al,a\rangle_E=\langle\be,a\rangle_E=0. \]
			The sectional curvature of $(E^{(r)},h)$ at $P$ is given by, 
			\begin{eqnarray*}
				K(P)
				&=&\langle R^M(X,Y)X,Y\rangle_{TM}+\frac1{r^2}|\al|^2|\be|^2+3\langle R^{\na^E}(X,Y)\al,\be\rangle_E-\frac34\langle R^{\na^E}(X,Y)a,R^{\na^E}(X,Y)a\rangle_E\\&&+ \frac14\sum_{i=1}^n\left(\langle R^{\na^E}(X,X_i)\be,a\rangle_E +\langle R^{\na^E}(Y,X_i)\al,a\rangle_E  \right)^2 -\sum_{i=1}^n
				\langle R^{\na^E}(X,X_i)\al,a\rangle_E\langle R^{\na^E}(Y,X_i)\be,a\rangle_E \\&&	+\langle \na_Y^{M,E}(R^{\na^E})(X,Y,\al)-\na_X^{M,E}(R^{\na^E})(X,Y,\be),a\rangle_E,
			\end{eqnarray*}\end{enumerate}where $(X_i)_{i=1}^n$ is any   orthonormal basis of $T_xM$.
		\end{pr}
		\begin{proof} If the rank of $E$ is equal to 2 then $\dim T_{(x,a)}E^{(r)}=n+1$ and $P\cap\{X^h,X\in T_xM  \}\not=0$ and hence $P$ contains a unitary vector $Y^h$. We take a unit vector $X^h+\al^t$ orthogonal to $Y^h$ to get a basis $(X^h+\al^t,Y^h)$ of $P$.
			
			If $\mathrm{rank}(E)>2$ we take an orthonormal basis $(X^h+\al^t,Y^h+\be^t)$ of $P$, i.e,
			\[ |X|^2+|\al|^2=|Y|^2+|\be|^2=1,\;\langle X,Y\rangle_{TM}+\langle \al,\be\rangle_E=0\esp \langle \al,a\rangle_E=\langle \be,a\rangle_E=0.  \]
			We suppose that $\langle X,Y\rangle_{TM}\not=0$ and  write
			$(\frac12(|X|^2-|Y|^2),\langle X,Y\rangle_{TM})=\rho(\cos(\mu),\sin(\mu))$ with $\mu\in[0,\frac{\pi}{2})$.
Then			  the vectors
			\[ U=\cos\left(\frac{\mu}2\right) (X^h+\al^t)+\sin\left(\frac{\mu}2\right) (Y^h+\be^t)\esp
			V=-\sin\left(\frac{\mu}2\right) (X^h+\al^t)+\cos\left(\frac{\mu}2\right) (Y^h+\be^t) \]constitute a basis of $P$ satisfying the desired relations.

			Let us compute the sectional curvature at $P$. We denote by $R$ the curvature tensor of $(E^{(r)},h)$.
			\begin{eqnarray*}
				K(P)&=&h(R(X^h+\al^t,Y^h+\be^t)(X^h+\al^t),Y^h+\be^t)\\
				&=&h(R(X^h+\al^t,Y^h+\be^t)X^h,Y^h)+h(R(X^h+\al^t,Y^h+\be^t)X^h,\be^t)+
				h(R(X^h+\al^t,Y^h+\be^t)\al^t,Y^h)\\&&+h(R(X^h+\al^t,Y^h+\be^t)\al^t,\be^t)\\
				&=&h(R(X^h,Y^h)X^h,Y^h)+h(R(X^h,\be^t)X^h,Y^h)+h(R(\al^t,Y^h)X^h,Y^h)+h(R(\al^t,\be^t)X^h,Y^h)\\
				&&+h(R(X^h,Y^h)X^h,\be^t)+h(R(X^h,\be^t)X^h,\be^t)+h(R(\al^t,Y^h)X^h,\be^t)+h(R(\al^t,\be^t)X^h,\be^t)\\
				&&+h(R(X^h,Y^h)\al^t,Y^h)+h(R(X^h,\be^t)\al^t,Y^h)+h(R(\al^t,Y^h)\al^t,Y^h)+
				h(R(\al^t,\be^t)\al^t,Y^h)\\
				&&+h(R(X^h,Y^h)\al^t,\be^t)+h(R(X^h,\be^t)\al^t,\be^t)+h(R(\al^t,Y^h)\al^t,\be^t)+
				h(R(\al^t,\be^t)\al^t,\be^t)\\
				&=&h(R(X^h,Y^h)X^h,Y^h)+2h(R(X^h,Y^h)X^h,\be^t)+2h(R(X^h,Y^h)\al^t,Y^h)+2h(R(X^h,Y^h)\al^t,\be^t)\\
				&&+h(R(X^h,\be^t)X^h,\be^t)+2h(R(\al^t,Y^h)X^h,\be^t)+2h(R(\al^t,\be^t)X^h,\be^t)\\
				&&+h(R(\al^t,Y^h)\al^t,Y^h)+
				2h(R(\al^t,\be^t)\al^t,Y^h)
				+
				h(R(\al^t,\be^t)\al^t,\be^t).
			\end{eqnarray*} Recall that the projection $\pi_E:(E^{(r)},h)\too (M,\prsm)$ is a Riemannian submersion with totally geodesic fibers and 
			O'Neill shape tensor  $B$ is given by \eqref{b}. So we can use O'Neill's formulas for  curvature given in \cite[chap. 9\;pp.241]{bes}. From these formulas we have $h(R(\al^t,\be^t)X^h,\be^t)=
			h(R(\al^t,\be^t)\al^t,Y^h)=0$ and hence
			\begin{eqnarray*}
				K(P)&=&h(R(X^h,Y^h)X^h,Y^h)+h(R(X^h,\be^t)X^h,\be^t)+h(R(\al^t,Y^h)\al^t,Y^h)+	h(R(\al^t,\be^t)\al^t,\be^t)\\&&+2h(R(X^h,Y^h)X^h,\be^t)+2h(R(X^h,Y^h)\al^t,Y^h)
				+2h(R(X^h,Y^h)\al^t,\be^t)
				+2h(R(\al^t,Y^h)X^h,\be^t).\end{eqnarray*}
			Let us give every term in this expression by using O'Neill's formulas and Proposition \eqref{f}.
			\begin{eqnarray*}
			h(R(X^h,Y^h)X^h,Y^h)&=&	\langle R^M(X,Y)X,Y\rangle_{TM}-\frac34\langle R^{\na^E}(X,Y)a,R^{\na^E}(X,Y)a\rangle_E,\\ 
			h(R(X^h,\be^t)X^h,\be^t)&=&h((\overline{\na}_{\be^t}B)_{X^h}X^h,\be^t)+
			h(B_{X^h}\be^t,B_{X^h}\be^t)=h(B_{X^h}\be^t,B_{X^h}\be^t),\\
			h(R(\al^t,Y^h)\al^t,Y^h)&=&h((\overline{\na}_{\al^t}B)_{Y^h}Y^h,\al^t)+
			h(B_{Y^h}\al^t,B_{Y^h}\al^t)=h(B_{Y^h}\al^t,B_{Y^h}\al^t),\\
			h(R(\al^t,\be^t)\al^t,\be^t)&=&\frac1{r^2}|\al|^2|\be|^2,\\
			2h(R(X^h,Y^h)X^h,\be^t)&=&2h((\overline{\na}_{X^h}B)_{X^h}Y^h),\be^t)=
			-
			\langle \na_X^{M,E}(R^{\na^E})(X,Y,\be),a\rangle_E,\\
			2h(R(X^h,Y^h)\al^t,Y^h)&=&-2h((\overline{\na}_{Y^h}B)_{X^h}Y^h),\al^t)
			=
			\langle \na_Y^{M,E}(R^{\na^E})(X,Y,\al),a\rangle_E,\\
			2h(R(X^h,Y^h)\al^t,\be^t)&=&2h((\overline{\na}_{\al^t}B)_{X^h}Y^h,\be^t)-2h((\overline{\na}_{\be^t}B)_{X^h}Y^h,\al^t)+
			2h(B_{X^h}\al^t,B_{Y^h}\be^t)-2h(B_{X^h}\be^t,B_{Y^h}\al^t)\\
			&=&2\langle R^{\na^E}(X,Y)\al,\be\rangle_E-2h(B_{X^h}\al^t,B_{Y^h}\be^t)+2h(B_{X^h}\be^t,B_{Y^h}\al^t)\\
			2h(R(\al^t,Y^h)X^h,\be^t)&=&-2h(R(X^h,\be^t)Y^h,\al^t)=
			-2h((\overline{\na}_{\be^t}B)_{X^h}Y^h,\al^t)-2h(B_{Y^h}\be^t,B_{X^h}\al^t)\\
			&=&\langle R^{\na^E}(X,Y)\al,\be\rangle_E-2h(B_{X^h}\al^t,B_{Y^h}\be^t).
			\end{eqnarray*}To complete the proof, we need to compute the quantity
			\[ Q=h(B_{X^h}\be^t,B_{X^h}\be^t)+h(B_{Y^h}\al^t,B_{Y^h}\al^t)
			-4h(B_{X^h}\al^t,B_{Y^h}\be^t)+2h(B_{Y^h}\al^t,B_{X^h}\be^t). \]
			When $E=TM$, $\prs_E=\prsm$ and $\na^E=\na^M$, one can use the formula \eqref{tm} to recover the expression of the sectional curvature given in
			\cite{kowal}.
		In the general case, we use instead \eqref{bxalpha} and we get
			\begin{eqnarray*}
			Q&=&\frac14\sum_{i=1}^n\langle R^{\na^E}(X,X_i)\be,a\rangle_E^2+
			\frac14\sum_{i=1}^n\langle R^{\na^E}(Y,X_i)\al,a\rangle_E^2
			-\sum_{i=1}^n\langle R^{\na^E}(X,X_i)\al,a\rangle_E\langle R^{\na^E}(Y,X_i)\be,a\rangle_E\\
			&&+\frac12\sum_{i=1}^n\langle R^{\na^E}(Y,X_i)\al,a\rangle_E\langle R^{\na^E}(X,X_i)\be,a\rangle_E.
			\end{eqnarray*}
			This completes the proof.
			\end{proof}

		\begin{exem}\label{exem1} Let $M=S^2$ with its canonical metric $\prsm$,  $E=TM$ and $\na^E=\na^M$. Let us compute the sectional curvature of $(T^{(1)}M,h)$. According to Proposition \ref{ke}, if $P$ is a plan in $T_{(x,u)}T^{(1)}M$ then $P=\mathrm{span}\{X^h+Z^t,Y^h \}$ with $X,Y,Z\in T_xM$, $|X|^2+|Z|^2=|Y|^2=1$ and $\langle Z,u\rangle_{TM}=0$. 
			The curvature $R^M$ is given by
			$R^M(X,Y)Z=\langle X,Z\rangle_{TM}Y-\langle Y,Z\rangle_{TM}X$. Hence
		\begin{eqnarray*}
		K(P)&=&\langle R^{M}(X,Y)X,Y\rangle_{TM}-\frac34|R^M(X,Y)u|^2+\frac14|R^M(Z,u)Y|^2\\
		&=&|X|^2-\frac34\left( \langle X,u\rangle_{TM}^2+\langle Y,u\rangle_{TM}^2|X|^2 \right)+\frac14\left(\langle Z,Y\rangle_{TM}^2+\langle u,Y\rangle_{TM}^2|Z|^2 \right).
		\end{eqnarray*}	If $Z=0$ then $K(P)=\frac14$. If $Z\not=0$ then $\{Z,u \}$ becomes an orthogonal basis of $T_xM$ 	and
		\[ 1=|Y|^2=\langle Y,u\rangle_{TM}^2+\frac1{|Z|^2}\langle Y,Z\rangle_{TM}^2. \]
		Thus
		\[ K(P)=|X|^2+\frac14|Z|^2-\frac34\left( \langle X,u\rangle_{TM}^2+\langle Y,u\rangle_{TM}^2|X|^2 \right). \]
		If $X=0$ then $K(P)=\frac14$. If $X\not=0$ then $\{X,Y\}$ is an orthogonal basis and hence
		\[ 1=|u|^2=\langle Y,u\rangle_{TM}^2+\frac1{|X|^2}\langle X,u\rangle_{TM}^2 \]and hence $K(P)=\frac14$. So $(T^{(1)}M,h)$ has constant sectional curvature $\frac14$. This has been proved first in \cite{nagy}.\end{exem}
		
		\begin{pr}\label{prricci} Let $X,Y \in T_xM$, $\al, \be \in E_x$ and $(x,a) \in E^{(r)}$ and $(X_i)_{i=1}^n$  any   orthonormal basis  of $T_xM$. Then:
			\begin{enumerate}
		\item  The Ricci curvature of $(E^{(r)},h)$ is given by
		\begin{eqnarray*}
			\ric(X^h+\al^t,Y^h+\be^t)&=&\frac{(m-2)}{r^2}\langle \overline{\al},\overline{\be}\rangle_E+\ric^M(X,Y)-\frac12\sum_{i=1}^n\langle R^{\na^E}(X,X_i)a,R^{\na^E}(Y,X_i)a\rangle_E\\
			&&-\frac12\sum_{i=1}^n\langle \na^{M,E}_{X_i}(R^{\na^E})(X_i,X,\be)+\na^{M,E}_{X_i}(R^{\na^E})(X_i,Y,\al),a\rangle_E\\
			&&+\frac14\sum_{i=1}^n\sum_{j=1}^n\langle R^{\na^E}(X_i,X_j)a,\al\rangle_E\langle R^{\na^E}(X_i,X_j)a,\be\rangle_E.
		\end{eqnarray*}
		
		\item	The scalar curvature of $(E^{(r)},h)$  is given by
		\[ \tau^r(x,a)=s^M(x)+\frac1{r^2}(m-1)(m-2)-\frac14\xi_x(a,a), \]
	\end{enumerate} where
	$$\xi_x(a,b)=\sum_{j=1}^n \sum_{i=1}^n\langle R^{\na^E}(X_i,X_j)a,R^{\na^E}(X_i,X_j)b\rangle_E,\quad a,b\in E_x.$$
\end{pr}

\begin{proof} We will use the O'Neil formulas for the Ricci curvature and scalar curvature given in \cite[Proposition 9.36, Corollary 9.37]{bes}. From these formulas, Proposition \ref{f} and the fact that the fibers are Einstein, we get
	\begin{eqnarray*}
	\ric(X^h,Y^h)&=&\ric^M(X,Y)-2\sum_{i=1}^nh(B_{X^h}X_i^h,B_{Y^h}X_i^h)=
	\ric^M(X,Y)-\frac12\sum_{i=1}^n\langle R^{\na^E}(X,X_i)a,R^{\na^E}(Y,X_i)a\rangle_E,\\
	\ric(\al^t,\be^t)&=&\frac{(m-2)}{r^2}\langle \overline{\al},\overline{\be}\rangle_E
	+\sum_{i=1}^nh(B_{X_i^h}\al^t,B_{X_i^h}\be^t)\\
	&=&\frac{(m-2)}{r^2}\langle \overline{\al},\overline{\be}\rangle_E+\frac14\sum_{i=1}^n\sum_{j=1}^n\langle R^{\na^E}(X_i,X_j)a,\al\rangle_E\langle R^{\na^E}(X_i,X_j)a,\be\rangle_E,\\
	\ric(X^h,\be^t)&=&-h(\check{\de}{B}X^h,\be^t)=
	\sum_{i=1}^nh((\overline{\na}_{X_i^h}B)_{X_i^h}X,\be^t)=-\frac12\sum_{i=1}^n\langle \na^{M,E}_{X_i}(R^{\na^E})(X_i,X,\be),a\rangle_E.
	\end{eqnarray*}This establish the expression of the Ricci curvature. The scalar curvature is given by $\tau^r=s^M\circ\pi_E+s^v+|B|^2$ which completes the proof.
\end{proof}

\section{On the sign of the different curvatures of $(E^{(r)},h)$}\label{section3}
In this section, we study the sign of sectional, Ricci and scalar curvature) of sphere bundles $E^{(r)}$ equipped with the Sasaki metric $h$.

Through this section,  $(M,\prsm)$ is a Riemannian manifold of dimension $n$ and $(E,\prs_E)$ is a Euclidean vector bundle of rank $m$ with an invariant connection $\na^E$.

\subsection{The case $R^{\na^E}=0$}

Note that $R^{\na^E}=0$ if and only if the O'Neill shape tensor of the Riemannian submersion $\pi_E:(E^{(r)},h)\too (M,\prsm)$ vanishes which is equivalent to $E^{(r)}$ being locally the Riemannian product of $M$ and the fiber. So we have the following results.
\begin{pr} Suppose $R^{\na^E}=0$ and $m=2$. Then, by using the notations in Propositions \ref{ke} and  \ref{prricci}
\[ K(P)=\langle R^M(X,Y)X,Y\rangle_{TM},\;\ric(X^h+\al^t,Y^h+\be^t)=\ric^M(X,Y)\esp \tau(x,a)=s^M(x). \]	
	
\end{pr}

\begin{pr}\label{3} Suppose $R^{\na^E}=0$ and $m\geq3$. Then
\begin{enumerate}
	\item  $(M,\prsm)$ has constant scalar curvature if and only if  $(E^{(r)},h)$ has constant scalar curvature,
	\item  $(M,\prsm)$ is locally symmetric if and only if  $(E^{(r)},h)$ is locally symmetric,
	\item  $(M,\prsm)$ is Einstein with Einstein constant $\frac{m-2}{r^2}$  if and only if  $(E^{(r)},h)$ is Einstein with the same Einstein constant,
	\item  $(E^{(r)},h)$ can never have a constant sectional curvature.
\end{enumerate}\end{pr}

 For the Euclidean vector bundles with large rank compared to the dimension of the base, the following theorem constitutes a converse to the third assertion in Proposition \ref{3}. Note that the rank of the Atiyah vector bundle $E(M,k)$ is $\frac{n(n+1)}2$ and hence it satisfies the hypothesis of the next theorem.

\begin{theo}\label{einstein} Suppose that $m-1>\frac{n(n-1)}2$ where $m$ is the rank of $E$ and $n=\dim M$. Then:\begin{enumerate}
		\item $(E^{(r)},h)$ is Einstein with Einstein constant $\la$ if and only if $R^{\na^E}=0$, $\la=\frac{(m-2)}{r^2}$ and $M$ is Einstein with Einstein constant $\frac{(m-2)}{r^2}$.
		\item $(E^{(r)},h)$ can never has constant sectional curvature.
	\end{enumerate}

\end{theo}

\begin{proof}\begin{enumerate}
		\item 
	 If $(E^{(r)},h)$ is Einstein then, according to Proposition \ref{prricci}, we have for any $x\in M$, $X\in T_xM$, $a\in E_x^{(r)}$ and $\al\in E_x$ with $\langle\al,a\rangle_E=0,$
	\begin{eqnarray}
	\label{eq2b} \la|\al|^2&=&\frac{(m-2)}{r^2}|\al|^2+\frac14\sum_{i=1}^n\sum_{j=1}^n\langle R^{\na^E}(X_i,X_j)a,\al\rangle_E^2.
	\end{eqnarray}Fix $x\in M$,  $a\in E_x^{(r)}$ and an orthonormal basis $(X_i)$ of $T_xM$ and choose an orthonormal family $(\al_1,\ldots,\al_{m-1})$ of elements in the orthogonal of $a$. For any $k=1,\ldots,m-1$ define the vector $U_k\in\R^{\frac{n(n-1)}2}$ by putting
	\[ U_k= \left( \langle R^{\na^E}(X_1,X_2)a,\al_k\rangle_E,R^{\na^E}(X_1,X_3)a,\al_k\rangle_E,\ldots, 
	\langle R^{\na^E}(X_{n-1},X_n)a,\al_k\rangle_E    \right).\] 
	If we take $\al=\al_k$ in \eqref{eq2b}, we get that the Euclidean norm of $U_k$ satisfies $|U_k|^2=2\left(\la-\frac{(m-2)}{r^2} \right)$. Moreover, if we take $\al=\al_k+\al_l$ with $l\not=k$ we get that $\langle U_l,U_k\rangle=0$. Thus $(U_1,\ldots,U_{m-1})$ is an orthogonal family of vector in $\R^{\frac{n(n-1)}2}$. Since $m-1>\frac{n(n-1)}2$ they must be linearly dependent. But they have the same norm so they must vanish. This completes the proof of the first assertion.
	\item If $(E^{(r)},h)$ has a constant sectional curvature then it is Einstein and hence $R^{\na^E}=0$. But, according to the expression of the sectional curvature given in Proposition \ref{ke} it cannot be constant. This completes the proof.\qedhere
	
\end{enumerate}
	\end{proof}
	
	\subsection{The case $\na^{M,E}(R^{\na^E})=0$}

If $\na^{M,E}(R^{\na^E})=0$ then $R^{\na^E}$ is invariant under parallel transport of $\na^M$ and $\na^E$ and hence
  there exists a constant $\mathbf{K}>0$ such that for any $X,Y\in\Ga(TM)$, $\al\in\Ga(E)$,
\begin{equation}\label{est} |R^{\na^E}(X,Y)\al|\leq \mathbf{K}|X||Y||\al|. \end{equation}The following theorem generalize  a result obtained in \cite{kowal}.
\begin{theo}\label{thek}
	 Suppose that $\na^{M,E}(R^{\na^E})=0$ and the sectional curvature of $M$ is bounded below by a positive constant $C$.  Then
	\begin{enumerate}
		\item The sectional curvature of $(E^{(r)},h)$ can never be nonpositive.
		
		\item If $\mathrm{rank}(E)=2$, then the sectional curvature of $(E^{(r)},h)$ is nonnegative if  $r^2\leq\frac{4C}{3\mathbf{K}}$.
		\item If $\mathrm{rank}(E)\geq 3$, then the sectional curvature of $(E^{(r)},h)$ is nonnegative if 
		\begin{equation}\label{eqcurv1}
		C-\frac34r^2\mathbf{K}^2\left(4+3r^2(n-2)\mathbf{K}+\frac34r^4(n-2)^2\mathbf{K}^2
		\right)\geq0.
		\end{equation}	
	\end{enumerate}In particular, for $r$ sufficiently small the sectional curvature of $(E^{(r)},h)$ is nonnegative.
\end{theo}
\begin{proof}
	Let $P\subset T_{(x,a)}E^{(r)}$ be a plane. Then there exists an orthonormal basis $\{ X^h+\al^t,Y^h+\be^t  \}$ of $P$ satisfying $|X|^2+|\al|^2=|Y|^2+|\be|^2=1,\;\langle X,Y\rangle_{TM}=\langle\al,\be\rangle_E=0$ and $ \langle\al,a\rangle_E=\langle\be,a\rangle_E=0. $ Put $X=\cos(t)\wi X$, $\al=\sin(t)\wi\al$, $Y=\cos(s)\wi Y$, $\be=\sin(s)\wi\be$ and $a=r\wi a$ with $s,t\in[0,\pi/2]$ and $|\wi X|=|\wi Y|=|\wi\al|=|\wi\be|=1$. We replace in the expression of $K(P)$ given in Proposition \ref{ke} and we get
	\begin{eqnarray*}
		K(P)&=&A\cos^2(t)\cos^2(s)+\frac1{r^2}\sin^2(t)\sin^2(s)+B\cos(t)\cos(s)\sin(t)\sin(s)
		+D\cos^2(t)\sin^2(s)\\&&+E\sin^2(t)\cos^2(s),
	\end{eqnarray*}where
	\begin{eqnarray*}
		A&=&K^M(\{\wi X,\wi Y \})-\frac34 r^2| R^{\na^E}(\wi X,\wi Y)\wi a|^2,\\
		B&=&3\langle R^{\na^E}(\wi X,\wi Y)\wi\al,\wi \be\rangle_E
		-r^2\sum_{i=1}^n
		\langle R^{\na^E}(\wi X,X_i)\wi \al,\wi a\rangle_E\langle R^{\na^E}(\wi Y,X_i)\wi \be,\wi a\rangle_E\\
		&&+\frac{r^2}2\sum_{i=1}^n\langle R^{\na^E}(\wi X,X_i)\wi \be,\wi a\rangle_E \langle R^{\na^E}(\wi Y,X_i)\wi \al,\wi a\rangle_E,\\
		D&=&\frac{r^2}4\sum_{i=1}^n\langle R^{\na^E}(\wi X,X_i)\wi \be,\wi a\rangle_E^2,\quad
		E=\frac{r^2}4\sum_{i=1}^n\langle R^{\na^E}(\wi Y,X_i)\wi \al,\wi a\rangle_E^2.
	\end{eqnarray*}
	\begin{enumerate}
		\item If $\cos(t)=\cos(s)=0$ then $K(P)=\frac1{r^2}>0$ and hence sectional curvature of $(E^{(r)},h)$ can never be nonpositive.
		
	\end{enumerate}

	Let us prove now the second and the third assertion.
	If $X=0$ or $Y=0$ then $K(P)\geq0$. Suppose now that $X\not=0$ and $Y\not=0$, so we can choose $X_1=\wi X$ and $X_2=\wi Y$ and get 
	 $$A\geq  C-\frac34r^2\mathbf{K}^2\esp B\geq-\frac{3\mathbf{K}}{2}\left(2+r^2(n-2)\mathbf{K}\right).$$
	 \begin{enumerate}

	\item[2.] If $\mathrm{rank}(E)=2$, we can choose $\be=0$ and hence  $$K(P)\geq (C-\frac34r^2\mathbf{K})\cos^2(t)\cos^2(s)+\frac1{r^2}\sin^2(t)\sin^2(s).$$ Thus
	the sectional curvature is nonnegative if $r^2\leq\frac{4C}{3\mathbf{K}}$.

	 \item[3.] Suppose that $\mathrm{rank}(E)>2$. Then, by using the estimations of $A$ and $B$ given above, we get
	 \[ K(P)\geq \left(C-\frac34r^2\mathbf{K}^2\right)\cos^2(t)\cos^2(s)+\frac1{r^2}\sin^2(t)\sin^2(s)-\frac{3\mathbf{K}}{2}\left(2+r^2(n-2)\mathbf{K}\right)\cos(t)\cos(s)\sin(t)\sin(s).  \]
	 The right side of this inequality, say $Q$, can be arranged in the following way:
	 \begin{eqnarray*}
	 Q&=&\left[\frac1{r}\sin(t)\sin(s)-\frac{3r\mathbf{K}}{4}\left(2+r^2(n-2)\mathbf{K}\right)\cos(t)\cos(s) \right]^2\\
	 &&+\left( C-\frac34r^2\mathbf{K}^2\left(4+3r^2(n-2)\mathbf{K}+\frac34r^4(n-2)^2\mathbf{K}^2
	   \right)\right)\cos^2(t)\cos^2(s).
	 \end{eqnarray*}This ends the proof of the last assertion.\qedhere

\end{enumerate}

\end{proof}

\begin{remark}\label{remark}\begin{enumerate}
		\item In the classical case, i.e., $E=TM$, $\prs_E=\prsm$ and $\na^E=\na^M$ the hypotheses $\na^M(R^M)=0$ and $M$ has positive sectional curvature imply that the sectional curvature of $M$ is bounded bellow by a positive constant. Thus, in this case our result is the same as the result obtained in \cite{kowal}.
		\item The left side of the inequality \eqref{eqcurv1}, say $Q$, goes to $C$ when $r$ goes to 0 which permitted as to get our result. In some cases the constant $\mathbf{K}$ can depend on a parameter and by varying this parameter one can make $Q>0$. This is the case in Theorem \ref{theo1}.
		
	\end{enumerate}

\end{remark}

\begin{theo}\label{thricci}
	 
	 Suppose that $\na^{M,E}(R^{\na^E})=0$ and $R^{\na^E}\not=0$ and  there exists a positive constant $\rho$ such that $\ric^M(X,X)\geq\rho|X|^2$ for any $X\in\Ga(TM)$. Then:
		\begin{enumerate}
			\item If $\mathrm{rank}(E)=2$ then $(E^{(r)},h)$ has nonnegative Ricci curvature for $r^2\leq\frac{2\rho}{n\mathbf{K}^2}$, where the constant $\mathbf{K}$ is given in \eqref{est}.
			\item If $\mathrm{rank}(E)>2$ then $(E^{(r)},h)$ has positive Ricci curvature for $r^2<\frac{2\rho}{n\mathbf{K}^2}$, where the constant $\mathbf{K}$ is given in \eqref{est}.
			
		\end{enumerate}

\end{theo}
\begin{proof} For any $x\in M$, $a\in E_x^{(r)}$, $X\in T_xM$ and $\al\in E_x$ such that $|X|^2+|\al|^2=1$ and $\langle\al,a\rangle_E=0$, we have from Proposition \ref{prricci} that
	\begin{eqnarray*}
		\ric(X^h+\al^t,X^h+\al^t)&=&\frac{(m-2)}{r^2}|\al|^2+\ric^M(X,X)-\frac12\sum_{i=1}^n| R^{\na^E}(X,X_i)a|^2\\
		&&-\sum_{i=1}^n\langle \na^{M,E}_{X_i}(R^{\na^E})(X_i,X,\al),a\rangle_E
		+\frac14\sum_{i=1}^n\sum_{j=1}^n\langle R^{\na^E}(X_i,X_j)a,\al\rangle_E^2.
	\end{eqnarray*}Let us write $X=\mathrm{cos}(t)\hat{X}$, $\al=\mathrm{sin}(t)\hat{\al}$ and $\hat{a}=a/r$ where $\hat{X}$ and $\hat{\al}$ are unit vectors.

		Suppose that $\na^{M,E}(R^{\na^E}) = 0.$ We obtain
		\begin{eqnarray*}
			\ric(X^h+\al^t,X^h+\al^t)&=& \mathrm{cos}^2(t)\left(
			\ric^M(\hat{X},\hat{X})-\frac{r^2}{2}\sum_{i=1}^n|R^{\na^E}(\hat{X},X_i)\hat{a}|^2\right)\\&&
			+\mathrm{sin}^2(t)\left(\frac{(m-2)}{r^2}+\frac{r^2}{4}
			\sum_{i=1}^n\sum_{j=1}^n\langle R^{\na^E}(X_i,X_j)\hat{a},\hat{\al}\rangle_E^2\right).
		\end{eqnarray*}
		From the hypothesis on $\ric^M$ and \eqref{est}, we get \[ 	\ric(X^h+\al^t,X^h+\al^t)\geq \left(\rho-\frac{nr^2\mathbf{K}^2}{2}\right)\mathrm{cos}^2(t)+\frac{(m-2)}{r^2}\mathrm{sin}^2(t).\]
		This shows the two assertions.

\end{proof}

\subsection{Ricci and scalar curvatures}

The  two following theorems are a generalization of  \cite[Theorem 3, Theorem 1]{kowal} established in the case when $E=TM$.

\begin{theo}\label{thriccib} If $M$ is compact with positive  Ricci curvature and $\mathrm{rank}(E)\geq3$, then for $r$ sufficiently small the Ricci curvature of the sphere bundle $(E^{(r)},h)$  is positive.
	
\end{theo}

\begin{proof} Suppose now that $M$ is compact with positive  Ricci curvature and put $X=\cos(t)\hat{X}$, $\al=\sin(t)\hat{\al}$ and $\hat{a}=\frac{a}r$ where $\hat{X} \in T_xM$, $\hat{\al}\in E_x$, $|\hat{X}|=|\hat{\al}|=1$ and $(x,a)\in E^{(r)}$.
	We have
	\begin{eqnarray*}
		\ric(X^h+\al^t,X^h+\al^t)&=& \mathrm{cos}^2(t)\;\ric^M(\hat{X},\hat{X})+\frac{(m-2)}{r^2}\mathrm{sin}^2(t)-\frac12r^2\mathrm{cos}^2(t)\sum_{i=1}^n| R^{\na^E}(\hat{X},X_i)\hat{a}|^2\\
		&&-r\mathrm{cos}(t)\mathrm{sin}(t)\sum_{i=1}^n  \langle \na^{M,E}_{X_i}(R^{\na^E})(X_i,\hat{X})\hat{\al},\hat{a}\rangle_E +\frac14
		\sum_{i=1}^n\sum_{j=1}^n\langle R^{\na^E}(X_i,X_j)a,\al\rangle_E^2,
		\\&& \geq \mathrm{cos}^2(t)\ric^M(\hat{X},\hat{X})+\frac{(m-2)}{r^2}\mathrm{sin}^2(t)-\frac12r^2\mathrm{cos}^2(t)\sum_{i=1}^n| R^{\na^E}(\hat{X},X_i)\hat{a}|^2\\&&-r\;\mathrm{cos}(t)\mathrm{sin}(t)\sum_{i=1}^n  \langle \na^{M,E}_{X_i}(R^{\na^E})(X_i,\hat{X})\hat{\al},\hat{a}\rangle_E .
	\end{eqnarray*}
	Since $M$ is compact, there exists positive constants $L_1$ and $L_2$ such that for any $x \in M$ and for any unit vectors $\hat{X},\hat{Y},\hat{Z} \in T_xM$ $\hat{\al},\hat{\be} \in E_x$,\[ 		|R^{\na^E}(\hat{X},\hat{Y})\hat{Z}| \leq L_1 \esp |\langle \na^{M,E}_{\hat{X}}(R^{\na^E})(\hat{Y},\hat{Z})\hat{\al},\hat{\be}\rangle_E| \leq L_2. \]
	On the other hand, there is a positive number $\epsilon$ such that $\ric^M(\hat{X},\hat{X})\geq\epsilon$ for every unit vector $\hat{X}$.
	Then, by using the above estimations, we get
	\begin{eqnarray*}
		\ric(X^h+\al^t,X^h+\al^t)&\geq& \mathrm{cos}^2(t)(\epsilon-\frac12r^2nL_1^2)+\frac{(m-2)}{r^2}\mathrm{sin}^2(t)-rnL_2\mathrm{cos}(t)\mathrm{sin}(t)\\
		&=&\left(\sqrt{A}\cos(t)-\frac{B}{2\sqrt{A}}\sin(t)\right)^2+C\sin^2(t),
		\end{eqnarray*}where  $A=\epsilon-\frac12r^2nL_1^2$, $B=rnL_2$, $C\left(\frac{m-2}{r^2}-\frac{B^2}{4{A}}\right)$ and $r$ taken such that $A,C>0$. Then, the right side of this inequality is  positive for every $t$.
\end{proof}

\begin{theo}\label{thscalar}
	Let $(M,\prsm)$ be a compact Riemannian manifold and $(E,\prs_E)$ be a Euclidean vector bundle with an invariant connection $\na^E$. Then for $r$ sufficiently small the scalar curvature of $(E^{(r)},h)$ is positive.
\end{theo}
\begin{proof} Suppose now that $M$ is compact and put $\hat{a}=\frac{a}r$ where $(x,a)\in E^{(r)}$.
	We have
	\[ \tau^r(x,a)=s^M(x)+\frac1{r^2}(m-1)(m-2)-\frac14r^2\xi_x(\hat{a},\hat{a}). \]
	Since $M$ is compact, there exists positive constants $L_1$ and $L_2$ such that for any $x \in M$ and for any unit vectors $X,Y \in T_xM$, $\al,\be \in E_x$ \[ |\langle R^M(X,Y)X,Y \rangle_{TM}| \leq L_1 \esp | R^{\na^E}(X,Y)\al| \leq L_2. \] 
	Then,
\[ 		\tau^r(x,a)  \geq \frac1{r^2}(m-1)(m-2) + \frac14n (n-1)( 4L_1-rL_2^2. )\]
	This means that $\tau^r$ is positive on $E^{(r)}$, when $r$ is sufficiently small.
\end{proof} Let $E\too M$ be a vector bundle. Recall that its associated sphere bundle is the quotient $S(E)=E/\sim$ where $a\sim b$ if there exists $t>0$ such that $a=tb$. Let $\prs_E$  be a Euclidean product on $E$. The associated $O(m)$-principal bundle has a connection so there exits a connection $\na^E$ on $E$ which preserves the metric $\prs_E$. Since $S(E)$ can be identified to $E^{(r)}$ for any $r$, by using Theorems \ref{thriccib} and \ref{thscalar} we get the following corollary which has  been proved in \cite{nash} by a different method.
\begin{co}
	Let $E\too M$ be a vector bundle over a compact Riemannian manifold and $S(E)\too M$ its associated sphere bundle. Then
	\begin{enumerate}
		\item If the Ricci curvature of $M$ is positive then $S(E)$ admits a complete Riemannian metric of positive curvature.
		\item $S(E)$ admits a complete Riemannian metric of positive scalar curvature.
	\end{enumerate}
\end{co}

We will end this section with a result which has been proved in \cite{book}
when $E=TM$, $\prsm=\prs_E$ and $\na^E$ is the Levi-Civita connection of $\prsm$.
\begin{theo}\label{scalar}
	Let $(M,\prsm)$ be a Riemannian manifold and $(E,\prs_E)$ a Euclidean vector bundle with an invariant connection $\na^E$. Then, the sphere bundle $(E^{(r)},h)$ equipped with the Sasaki metric has constant scalar curvature if and only if
	\begin{eqnarray} \label{s1} \xi&=&\frac{|R^{\na^E}|^2}{m}\prs_E,\\
	\label{s2} 4ms^M-r^2|R^{\na^E}|^2&=&\mbox{constant}.
	\end{eqnarray}
	where $ \xi(a,b)=\sum_{j=1}^n\left( \sum_{i=1}^n\langle R^{\na^E}(X_i,X_j)a,R^{\na^E}(X_i,X_j)b\rangle_E\right)$ for any $a,b \in \Ga(E)$.
\end{theo}
\begin{proof} The scalar curvature $\tau^r$ is giving by, for $(x,a) \in E^{(r)}$ \[ \tau^r(x,a) = s^M(x)+\frac{1}{r^2}(m-1)(m-2)-\frac{1}{4}\xi_x(a,a). \]
	Suppose that $\tau$ is constant along $E^{(r)}$. For fixed $x \in M$ , $\tau^r(x,a)$ does not depend on the choice of the vector $a\in E^{(r)}_x$. This implies that $\xi_x$ is proportional to the metric $\langle\;,\;\rangle_E$ and the coefficient of proportionality  is necessarily equal to $|R^{\na^E}|^2/m$.
\end{proof}

\section{Sasaki metric on the sphere bundle of the Atiyah Euclidean vector bundle associated to a Riemannian manifold} \label{section4}

We have seen in the last section that many results obtained on the sphere bundles of  tangent bundles over Riemannian manifolds can be generalized to any Euclidean vector bundle. In this section, we will express these results in the case of the sphere bundle of the Atiyah Euclidean vector bundle introduced in the introduction to get some new interesting geometric situations and to open new horizons for further explorations.

\subsection{The Atiyah Euclidean vector bundle and the supra-curvature of a Riemannian manifold }

  Let $(M,\prsm)$ be a Riemannian manifold, $k>0$ and $(E(M,k),\prs_k,\na^E)$ the  associated 
  Atiyah Euclidean vector bundle defined in the introduction.
 Let $K^M: \mathrm{so}(TM) \to \mathrm{so}(TM)$ be the curvature operator given by $K^M(X\wedge Y)=R^M(X,Y)$ where $X\wedge Y(Z)=\langle Y,Z \rangle_{TM} X-\langle X,Z \rangle_{TM} Y.$\\
	 The curvature $R^{\na^E}$ of $\na^E$ (we refer to as the supra-curvature of $(M,\prsm,k)$)  was computed in \cite[Theorem 3.1]{boucettaessoufi}.  It is given by the following formulas: 
	\begin{eqnarray}
		R^{\na^E}(X,Y)Z&=&\left\{R^M(X,Y)Z+H_{Y}H_{X}Z-H_{X}H_{Y}Z\right\}+\left\{ -\frac12 \na_Z^M(K^M)(X\wedge Y) \right\},\nonumber\\
		R^{\na^E}(X,Y)F&=&\left\{(R^{\na^E}(X,Y)F)_{TM} \right\}
		+\left\{  [R^{M}(X,Y),F]+H_{Y}H_{X}F-H_{X}H_{Y}F\right\},\label{r}\\
		\langle (R^{\na^E}(X,Y)F)_{TM},Z\rangle_{k}&=&-\langle R^{\na^E}(X,Y)Z,F\rangle_k,\nonumber
	\end{eqnarray}$X,Y,Z\in\Ga(TM)$, $F\in\Ga(\mathrm{so}(TM))$. We denote by $E^{(r)}(M,k)$ the sphere bundle of radius $r$ associated to $E(M,k)$ and $h$ the Sasaki metric on $E^{(r)}(M,k)$.

	The supra-curvature is deeply related to the geometry of $(M,\prsm)$. Let us compute it in some particular cases. This computation will be useful in the proof of Theorem \ref{supra} where we will characterize the Riemannian manifolds with vanishing supra-curvature.
	 
	 \paragraph{Supra-curvature of the Riemannian product of  Riemannian manifolds}

	 \begin{pr}\label{pr4} Let $(M,\prsm)$ be the Riemannian product of $p$ Riemannian manifolds $(M_1,\prs_1),\ldots,(M_p,\prs_p)$. Then the supra-curvature of $(M,\prsm)$ at a point $x=(x_1,\ldots,x_p)$ is given by
	 	\[\left\{ \begin{matrix} R^{\na^E}[(X_1,\ldots,X_p),(Y_1,\ldots,Y_p)](Z_1,\ldots,Z_p)=
	 	R^{\na^{E_1}}(X_1,Y_1)Z_1+\ldots+R^{\na^{E_p}}(X_p,Y_p)Z_p,\\
	 	R^{\na^E}[(X_1,\ldots,X_p),(Y_1,\ldots,Y_p)](F)=
	 	R^{\na^{E_1}}(X_1,Y_1)F_1+\ldots+R^{\na^{E_p}}(X_p,Y_p)F_p,\end{matrix} \right.
	 	\]	where $X_i,Y_i,Z_i\in T_{x_i}M_i$, $F\in\mathrm{so}(T_xM)$, $F_i=\mathrm{pr}_i \circ F_{|TM_i}$, $R^{\na^{E_i}}$ is the supra-curvature of $(M_i,\prs_i,k)$ and $i=1,\ldots,p$.

	 \end{pr}

	 \begin{proof} It is an immediate consequence of the following formulas
	 	\begin{eqnarray*} R^M[X,Y](Z)&=&
	 	(R^{M_1}(X_1,Y_1)Z_1,\ldots,R^{M_p}(X_p,Y_p)Z_p),\\
	 	H_XY&=&H_{X_1}^1Y_1+\ldots+H_{X_p}^pY_p,\\
	 	H_XF&=&H_{X_1}^1F_1+\ldots+H_{X_p}^pF_p,\\
	 	\na^M_X(K^M)(X\wedge Y)&=&\na_{Z_1}(K^{M_1})(X_1\wedge Y_1)+\ldots+\na_{Z_p}(K^{M_p})(X_p\wedge Y_p),
	 	 \end{eqnarray*}where $X=(X_1,\ldots,X_p)$, $Y=(Y_1,\ldots,Y_p)$,  $Z=(Z_1,\ldots,Z_p)$ and $F_i=\mathrm{pr}_i \circ F_{|TM_i}$.
	 	\end{proof}

\paragraph{Supra-curvature of Riemannian manifolds with constant curvature}

	\begin{pr}\label{pr3} Suppose that $(M,\prsm)$ has constant sectional curvature $c$ and put $\varpi=\frac14c(2-ck)$.  Then, for any $X,Y\in\Ga(TM)$ and $F\in\Ga(\mathrm{so}(TM))$,
		$$ R^{\na^E}(X,Y)Z=-2\varpi X\wedge Y(Z)\esp R^{\na^E}(X,Y)F=-2\varpi [X\wedge Y,F]. $$
		\end{pr}
	\begin{proof} The expression of $R^{\na^E}$ is given by \eqref{r}.  We have $H_XY=-\frac12R^M(X,Y)=\frac12cX\wedge Y$. Moreover, since the curvature is constant then $\na^M(K^M)=0$.
		
	Now if $(X_i)_{i=1}^n$ is a local frame of orthonormal vector fields then
		\begin{eqnarray*} 
			\langle H_XF,Y\rangle_{TM}&=&-\frac12k\;\tr( F\circ R^M(X,Y))
			=-\frac12ck\sum_{i=1}^n\langle F(X_i),X\wedge Y(X_i)\rangle_{TM} \\
			&=&-\frac12ck\sum_{i=1}^n\left(\langle Y,X_i\rangle_{TM}\langle F(X_i),X\rangle_{TM}-\langle X,X_i\rangle_{TM}\langle F(X_i),Y\rangle_{TM} \right)\\
			&=&-ck\langle F(Y),X\rangle_{TM}.
		\end{eqnarray*}Thus $H_XF=ckF(X)$. So
		\begin{eqnarray*}
			\;[H_Y,H_X]Z&=&\frac12(H_YR^M(Z,X)+H_XR^M(Y,Z))\\&=&\frac12ck(R^M(Z,X)Y+R^M(Y,Z)X)\\
			&=&-\frac12ckR^M(X,Y)Z.
		\end{eqnarray*}
		Thus
		\[ R^{\na^A}(X,Y)Z=\frac12(2-ck) R^M(X,Y)Z=-\frac12c(2-ck)X\wedge Y(Z).\]
		On the other hand,
		\begin{eqnarray*}
			\;[H_Y,H_X]F&=&ck(H_YF(X)-H_XF(Y))\\
			&=&-\frac12ck(R^M(Y,F(X))+R^M(F(Y),X)),\\
			&=&-\frac12c^2k([F,X\wedge Y]).
		\end{eqnarray*}This completes the proof. 
		\end{proof}

	\paragraph{Supra-curvature of some locally symmetric spaces}
Let $G$ be a compact connected Lie group with $\g$ its Lie algebra and $K$ be a closed subgroup of $G$ with $\kr$ its Lie algebra. Denote by $\pi:G\too G/K$ the canonical projection. Suppose that $\g=\kr\oplus\p$ with $\p$ is $\Ad_{K}$-invariant,  $[\p,\p]\subset\kr$ and the restriction of the Killing form $B$ of $\g$ to $p$ is negative definite. The scalar product $\prs_\p=\la B_{|\p\times \p}$ with $\la<0$ defines a $G$-invariant Riemannian metric $\prs_{G/K}$ on $G/K$ which is locally symmetric. For any $X\in\kr$, we denote by $\Phi_X$ the restriction of $\ad_X$ to $\p$, then 
\begin{equation}\label{s}
\mathrm{so}(\p,\prs_\p)=\Phi_{\kr}\oplus (\Phi_{\kr})^\perp,
\end{equation}where $(\Phi_{\kr})^\perp$ is the orthogonal with respect to the invariant scalar product on $\mathrm{so}(\p,\prs_\p)$,  $(A,B)\mapsto -\tr(AB)$. 
\begin{pr}\label{pr5} The supra-curvature of $(G/K,\prs_{G/K},k)$ at $\pi(e)$ is given by
	\begin{eqnarray*}
		R^{\na^E}(X,Y)Z
		&=&[[X,Y],Z]-\frac{k}4\left( [Y,U(\Phi_{[X,Z]})]-[X,U(\Phi_{[Y,Z]})]\right),\\
		R^{\na^E}(X,Y)F&=&[\Phi_{[X,Y]},\Phi_{X^F+\frac{k}4U(F)}]+[\Phi_{[X,Y]},F^\perp],
	\end{eqnarray*} where $X,Y,Z\in T_{\pi(e)}G/K=\p,$
	$F=\ad_{X^F}+F^\perp\in\mathrm{so}(\p,\prs_\p)=\Phi_{\kr}\oplus (\Phi_{\kr})^\perp$ and $U(F)$ is the element of $\kr$ given by
	\[ U(F)=\sum_{i=1}^n[X_i,F(X_i)], \]$(X_1,\ldots,X_n)$ an orthonormal basis of $\p$.
\end{pr}

\begin{proof} The expression of $R^{\na^E}$ is given by \eqref{r}. The curvature of $G/K$ at $\pi(e)$ is given by (see \cite[Proposition 7.72]{bes})
	\[ R^{G/K}(X,Y)Z=[[X,Y],Z],\quad X,Y,Z\in\p, \]and $\na^{G/K}(K^{G/K})=0$. Choose $(X_i)_{i=1}^n$ an orthonormal basis of $\p$.
	We have
	\begin{eqnarray*}
		\langle H_XF,Y\rangle_{k}&=&\langle H_XF,Y\rangle_{\p}\\
		&=&-\frac{k}2\sum_{i}\langle F(X_i),[[X,Y],X_i]\rangle_\p\\
		&=&\frac{\la k}2\sum_{i}B(F(X_i),[[X,Y],X_i])\\
		&=&\frac{k}2\sum_{i}\langle[X,[X_i,F(X_i)]],Y\rangle_\p.
	\end{eqnarray*}Thus
	$H_XF=\frac{k}2[X,U(F)]$. We deduce that
	\begin{eqnarray*}
	H_YH_XZ-H_XH_YZ&=&-\frac12H_Y(\Phi_{[X,Z]})+\frac12H_X(\Phi_{[Y,Z]})\\
	&=&-\frac{k}4[Y,U(\Phi_{[X,Z]})]+\frac{k}4[X,U(\Phi_{[Y,Z]})],\\
	H_YH_XF-H_XH_YF&=&-\frac{k}4\Phi_{[Y,[X,U(F)]]}+\frac{k}4\Phi_{[X,[Y,U(F)]]}\\
	&=&\frac{k}4[\Phi_{[X,Y]},\Phi_{U(F)}].
	\end{eqnarray*}This gives the desired formulas.
	\end{proof}

	\paragraph{Supra-curvature of complex projective spaces}
	Let $\pi:\mathbb{C}^{n+1}\setminus\{0\}\too P^n(\mathbb{C})$  be the natural projection  and $\pi_s:S^{2n+1}\too P^n(\mathbb{C})$ its restriction to $S^{2n+1}\subset\mathbb{C}^{n+1}\setminus\{0\}$. For any $m\in S^{2n+1}$, put $F_m=\ker((\pi_s)_*)_m$ and let $F_m^\perp$ be the orthogonal complementary subspace to $F_m$ in $T_m(S^{2n+1})$; $$T_m(S^{2n+1})=F_m\oplus F_m^\perp.$$We introduce the Riemannian metric $\prs_{P^n(\mathbb{C})}$ on $P^n(\mathbb{C})$ so that the restriction of $(\pi_s)_*$ to $F_m^\perp$ is an isometry onto $T_{\pi(m)}(P^n(\mathbb{C}))$. Let $J_0$ be the canonical complex structures on $\mathbb{C}^{n+1}$ and the standard complex structures $J$ on  $P^n(\mathbb{C})$ is given by $$ J(\pi_s)_*v=(\pi_s)_*J_0v,\; v\in F_m^\perp. $$
	 \begin{pr}\label{pr6} The curvature and the  supra-curvature of $(P^n(\mathbb{C}),g,k)$ are given by
	 	\begin{eqnarray*}
	 		R^{P^n(\mathbb{C})}(X,Y)Z&=&\langle X,Z	\rangle _{P^n(\mathbb{C})}
Y-\langle Y,Z	\rangle _{P^n(\mathbb{C})}X-2 \langle JY,X	\rangle _{P^n(\mathbb{C})}JZ+\langle JZ,Y	\rangle _{P^n(\mathbb{C})}JX- \langle JZ,X	\rangle _{P^n(\mathbb{C})}JY,\\
	 		R^{\na^E}(X,Y)Z&=&(k-1)\left( \langle Y,Z	\rangle _{P^n(\mathbb{C})}X-\langle X,Z\rangle _{P^n(\mathbb{C})}Y+2\langle JY,X	\rangle _{P^n(\mathbb{C})}JZ \right) \\&&+((2n+3)k-1)\left( \langle JZ,X	\rangle _{P^n(\mathbb{C})}JY-\langle JZ,Y	\rangle _{P^n(\mathbb{C})}JX\right),\\
	 		R^{\na^E}(X,Y)F&=& \left(\frac{k}2-1  \right) [F,  X\wedge Y+JX\wedge JY ]  +2 \langle JY,X	\rangle _{P^n(\mathbb{C})}
 [F,J]\\&&+ \frac{k}2 \left(  [J\circ F \circ J,X\wedge Y] -J\circ  F(X) \wedge JY -JX\wedge J\circ  F(Y)  \right),
	 	\end{eqnarray*}where $X,Y,Z\in\Ga(TP^n(\mathbb{C}))$ and $F\in \Ga(\mathrm{so}(TP^n(\mathbb{C})))$.
	 	
	 \end{pr}
	 
	 \begin{proof} The projection $\pi_s:S^{2n+1}\too P^n(\mathbb{C})$ is a Riemannian submersion with totally geodesic fiber and its O'Neill shape tensor is given by $A_{X^h}Y^h=-\langle J_0X^h,Y^h\rangle_{ \mathbb{C}^{n+1}}J_0N$ where $N$ is the radial vector field and $X^h,Y^h$ are the horizontal lift of $X,Y\in\Ga(P^n(\mathbb{C}))$. The expression of $R^{P^n(\mathbb{C})}$ follows from the formulas
	 	\begin{eqnarray*} \langle R^{S^{2n+1}}(X^h,Y^h)Z^h,T^h\rangle_{S^{2n+1}}&=&
	 	\langle R^{P^n(\mathbb{C})}(X,Y)Z,T\rangle_{P^n(\mathbb{C})}\circ\pi_s
	 	-2\langle A_{X^h}Y^h,A_{Z^h}T^h\rangle_{S^{2n+1}}\\&&+
	 	\langle A_{Y^h}Z^h,A_{X^h}T^h\rangle_{S^{2n+1}}-\langle A_{X^h}Z^h,A_{Y^h}T^h\rangle_{S^{2n+1}},\\
	 	R^{S^{2n+1}}(X^h,Y^h)Z^h&=&-(X^h\wedge Y^h)Z^h.
	 	 \end{eqnarray*}
	 To compute the supra-curvature, we use \eqref{r}. We choose an orthonormal frame $(X_i)_{i=1}^{2n}$ of $P^n(\mathbb{C})$. We have{
	 	\begin{eqnarray*}
	 		\langle H_XF,Y\rangle_{P^n(\mathbb{C})}&=&\frac{k}2\sum_{i=1}^{2n}\langle R^{P^n(\mathbb{C})}(X,Y)X_i,F(X_i)\rangle _{P^n(\mathbb{C})}  \\&=&\frac{k}2\sum_{i=1}^{2n}\left[ \langle X,X_i\rangle _{P^n(\mathbb{C})} \langle Y,F(X_i)\rangle _{P^n(\mathbb{C})} -\langle Y,X_i\rangle _{P^n(\mathbb{C})}\langle X,F(X_i)\rangle _{P^n(\mathbb{C})}-2\langle JY,X\rangle _{P^n(\mathbb{C})} \langle JX_i,F(X_i)\rangle _{P^n(\mathbb{C})}\right.\\&&\left.+ \langle JX_i,Y\rangle _{P^n(\mathbb{C})}
	 		\langle JX,F(X_i)\rangle _{P^n(\mathbb{C})}-\langle JX_i,X\rangle _{P^n(\mathbb{C})} \langle JY,F(X_i)\rangle _{P^n(\mathbb{C})} \right]\\
	 		&=&\frac{k}2\left(2 \langle F(X),Y\rangle _{P^n(\mathbb{C})}-2\tr(F\circ J)\langle JX,Y\rangle _{P^n(\mathbb{C})}-\langle JX,F(JY)\rangle _{P^n(\mathbb{C})}+\langle JY,F(JX) \rangle _{P^n(\mathbb{C})}\right).
	 	\end{eqnarray*}}Thus
	 	\[ H_XF=k(F(X)-\tr(F\circ J)JX-J\circ F\circ J(X)). \]
	 	So
	 	\begin{eqnarray*}
	 		H_YH_XZ&=&-\frac{k}2(R^{P^n(\mathbb{C})}(X,Z)Y-\tr(R^{P^n(\mathbb{C})}(X,Z)\circ J)JY-J\circ R^{P^n(\mathbb{C})}(X,Z)\circ J(Y))
	 	\end{eqnarray*}
	 	But $R^{P^n(\mathbb{C})}(X,Z)\circ J=J\circ R^{P^n(\mathbb{C})}(X,Z)$ and a direct computation gives that $\tr(J\circ R^{P^n(\mathbb{C})}(X,Y))=4(n+1)\langle JY,X\rangle _{P^n(\mathbb{C})}$.
	 	
	 	So
	 	\begin{eqnarray*}
	 		H_YH_XZ&=&k\left( 2(n+1)\langle JZ,X\rangle _{P^n(\mathbb{C})}JY- R^{P^n(\mathbb{C})}(X,Z)Y\right)\\
	 		&=&k\left( \langle Y,Z\rangle _{P^n(\mathbb{C})}X-\langle X,Y\rangle _{P^n(\mathbb{C})}Z-\langle JY,Z\rangle _{P^n(\mathbb{C})}JX+\langle JY,X\rangle _{P^n(\mathbb{C})}JZ +2(n+1)\langle JZ,X\rangle _{P^n(\mathbb{C})}JY \right).
	 	\end{eqnarray*}Thus	
	 	
\begin{eqnarray*}
	 		H_YH_XZ-H_XH_YZ	&=& k( \langle Y,Z\rangle _{P^n(\mathbb{C})}X-\langle X,Z\rangle _{P^n(\mathbb{C})}Y+2\langle JY,X\rangle _{P^n(\mathbb{C})}JZ \\&&+(2n+3)\left( \langle JZ,X\rangle _{P^n(\mathbb{C})}JY-\langle JZ,Y\rangle _{P^n(\mathbb{C})}JX\right)	).
	 	\end{eqnarray*}	 	
	 	Then
	 	 	\begin{eqnarray*}
	 		R^{\na^E}(X,Y)Z&=&(k-1)\left( \langle Y,Z\rangle _{P^n(\mathbb{C})}X-\langle X,Z\rangle _{P^n(\mathbb{C})}Y+2\langle JY,X\rangle _{P^n(\mathbb{C})}JZ \right) \\&&+((2n+3)k-1)\left( \langle JZ,X\rangle _{P^n(\mathbb{C})}JY-\langle JZ,Y\rangle _{P^n(\mathbb{C})}JX\right).
	 	\end{eqnarray*}
	 	On the other hand, 
	 		\begin{eqnarray*}
	 		H_YH_XF &=& k \left(  H_Y F(X)-\tr(F\circ J)H_YJX-H_Y J\circ F \circ J(X)  \right) 
	 		\\ &=& \frac{k}2  (  Y\wedge F(X)+ JY\wedge F \circ J(X)+JY\wedge J \circ F(X)-Y\wedge J \circ F \circ J(X)\\&& +2 \langle J\circ  F(X)-F\circ J(X),Y\rangle _{P^n(\mathbb{C})}J -\tr(F \circ J)\left( Y \wedge JX-JY \wedge X+2\langle X,Y\rangle _{P^n(\mathbb{C})}J \right) ).
	 	\end{eqnarray*}
	 	So, since $F(X)\wedge Y + X \wedge F(Y)=[F,X\wedge Y]$
	 	\[	H_YH_XF-H_XH_YF= \frac{k}2  \left( [  X\wedge Y+JX\wedge JY,F]+[J\circ F \circ J,X\wedge Y] -J\circ  F(X) \wedge JY -JX\wedge J\circ  F(Y) \right),\]
	 	and 
	 	\[	 		[R^{P^n(\mathbb{C})}(X,Y),F] = -[X\wedge Y +JX\wedge JY +2 \langle JY,XJ\rangle _{P^n(\mathbb{C})}  ,F] =- [  X\wedge Y +JX\wedge JY,F]-2 \langle JY,X\rangle _{P^n(\mathbb{C})} [J,F].\]
	 	Thus
	 	\begin{eqnarray*}
	 		R^{\na^E}(X,Y)F&=& [R^{P^n(\mathbb{C})}(X,Y),F] + H_YH_XF-H_XH_YF\\
	 		&=&(\frac{k}2-1  ) [F,  X\wedge Y+JX\wedge JY ] +2 \langle JY,X\rangle _{P^n(\mathbb{C})} [F,J] + \frac{k}2  [J\circ F \circ J,X\wedge Y] \\&&-\frac{k}2(J\circ  F(X) \wedge JY +JX\wedge J\circ  F(Y)  ).
	 	\end{eqnarray*}

	 \end{proof}

It is obvious that if $(M,\prsm)$ is flat then, for any $k>0$, the supra-curvature of $(M,\prsm,k)$ vanishes. Furthermore, according to Propositions \ref{pr4} and \ref{pr3}, if $(M,\prsm)$ is the Riemannian product of $p$ Riemannian manifolds all having constant sectional curvature $\frac{2}k$ then the supra-curvature of $(M,\prsm,k)$ vanishes. Actually, there are the only cases where the supra-curvature vanishes.

\begin{theo}\label{supra} Let $(M,\prsm)$ be a connected Riemannian manifold. Then the supra-curvature of $(M,\prsm,k)$ vanishes if and only if the Riemannian  universal  cover of $(M,\prsm)$ is isometric to $(\R^n,\prs_0)\times {\mathbb{S}}^{n_1}\left(\sqrt{\frac{k}2}\right)\times\ldots \times {\mathbb{S}}^{n_p}\left(\sqrt{\frac{k}2}\right)$ where ${\mathbb{S}}^{n_i}\left(\sqrt{\frac{k}2}\right)$ is the Riemannian sphere of dimension $n_i$, of radius $\sqrt{\frac{k}2}$ and constant curvature $\frac{2}k$.

\end{theo} 

\begin{proof} Suppose that the supra-curvature of $(M,\prsm,k)$ vanishes and consider the Riemannian covering $(N,\prs_{TN})$ of $(M,\prsm)$. Since $(M,\prsm)$ and
	$(N,\prs_{T N})$ are locally isometric then the supra-curvature of $(N,\prs_{TN},k)$ vanishes.  
	This implies by virtue of \eqref{r} that $(N,\prs_{TN})$ is locally symmetric
	and for any $X,Y\in\Ga(TN)$,
	\[ \langle R^N(X,Y)X,Y\rangle_{TN}=\langle H_XY,H_XY\rangle_k\geq0. \]Thus
	$(N,\prs_{TN})$ has non-negative sectional curvature. Since $N$ is simply-connected
	 then $(N,\prs_{TN})$ is a symmetric space. But a simply-connected symmetric space is the Riemannian product of a Euclidean space and a finite family of irreducible symmetric spaces (see \cite[Theorem 7.76]{bes}). Thus, $(N,\prs_{TN})=(E,\prs_0)\times (N_1,\prs_1)\times\ldots\times (N_p,\prs_p)$ where $(E,\prs_0)$ is flat and the $(N_i,\prs_i)$ are irreducible symmetric spaces with non-negative sectional curvature. This implies that the $N_i$ are compact and Einstein. According to Proposition \ref{pr4}, the vanishing of the supra-curvature of $(N,\prs_{TN},k)$ implies the vanishing of the supra-curvature of $(N_i,\prs_i,k)$ for $i=1,\ldots,p$. 
	 
	 Let $i\in\{1,\ldots,p \}$ and denote by $n_i$ the dimension of $N_i$. The symmetric space $N_i$ can be identified to $G/K$, where $G$ is the component of the identity of the group of isometries of $(N_i,\prs_i)$ and $K$ is the isotropy at some point. Moreover, the Lie algebra $\g$ of $G$ has a splitting $\g=\kr\oplus\p$ where $\kr$ is the Lie algebra of $K$ and $[\p,\p]\subset\kr$. Since $N_i$ is Einstein, the metric in restriction to $\p$ is proportional to the restriction of the Killing form. 
	  
	The vanishing of the supra-curvature of $(N_i,\prs_i,k)$ implies, by virtue of the second formula in  Proposition \ref{pr5},
	$[\Phi_{[\p,\p]},\Phi_{\kr}^\perp]=0$. This relation and the fact that $[\p,\p]$ is an ideal of $\kr$ imply that
	 $\Phi_{[\p,\p]}$ is an ideal of $\mathrm{so}(\p)$. But if $\dim\p\not=4$ then  the real Lie algebra $\mathrm{so}(\p)$ is simple  (see \cite[Theorem 6.105 ]{knapp}) and, in this case,
	    	$\Phi_{[\p,\p]}=0$ or $\Phi_{[\p,\p]}=\mathrm{so}(\p)$. If $\Phi_{[\p,\p]}=0$ then $R^{N_i}=0$ and we get the result. Otherwise, $\dim\kr\geq\dim\Phi_{\kr}\geq\dim\mathrm{so}(\p)=\frac{n_i(n_i-1)}2$. So
	\[ \dim G=\dim\kr+n_i\geq\frac{n_i(n_i+1)}2. \]But the dimension of the group of isometries is always less or equal to $\frac{n_i(n_i+1)}2$ with equality when the manifold has constant curvature. 
	Thus $\dim G=\frac{n(n+1)}2$ and hence $N_i$ has constant curvature. If $\dim\p=4$, $(N_i,\prs_i)$ is a Einstein four dimensional homogeneous space and   according to the main result in \cite[]{jensen}, $(N_i,\prs_i)$ is isometric to ${\mathbb{S}}^4(r)$, ${\mathbb{S}}^2(r)\times {\mathbb{S}}^2(r)$ or $P^2(\Co)$. But Proposition \ref{pr6} shows that the supra-curvature of $P^2(\Co)$ doesn't vanishes and Proposition \ref{pr3} shows that ${\mathbb{S}}^n(r)$ has vanishing supra-curvature if and only if $r={\sqrt{\frac{k}2}}$. This completes the proof.
	\end{proof}	
	
\subsection{Geometry of $(E^{(r)}(M,k),h)$ when $M$ is locally symmetric}	
The following proposition is a key step 	  in order to apply Theorems \ref{thek} and \ref{thricci} to $E(M,k)$. 
\begin{pr}\label{prls}
	If $M$ is locally symmetric then $\na^{M,E}(R^{\na^E})=0$.
\end{pr}
\begin{proof}
	Assume that $M$ is locally symmetric which is equivalent to $\na^M(K^M)=0$. Note first that $\na^{M,E}(R^{\na^E})=0$ if and only if for any curve $\ga:[a,b]\too M$,   $V_1,V_2,V_3:[a,b]\too TM$ parallel vector fields along $c$ and $F:[a,b]\too \mathrm{so}(TM)$  parallel section along $c$ then $R^{\na^E}(V_1,V_2)V_3$ and 
	$R^{\na^E}(V_1,V_2)F$ are parallel along $c$. But $R^M(V_1,V_2)V_3$ is parallel, $H_{V_1}V_2$ and $H_{V_1}F$ are also parallel and by using \eqref{r} we can conclude.
	\end{proof}

 The following theorem is an immediate consequence of Theorem \ref{thek}, Theorem \ref{thricci} and Proposition \ref{prls}.

\begin{theo}\label{theo2}\begin{enumerate}\item If $(M,\prsm)$ is locally symmetric and its sectional curvature is positive then, for $r$ sufficiently small, $(E^{(r)}(M,k),h)$ has nonnegative sectional curvature.
	
\item 	If $M$ is compact with positive  Ricci curvature or locally symmetric with positive Ricci curvature, then for $r$ sufficiently small the Ricci curvature of  $(E^{(r)}(M,k),h)$  is positive.
	\end{enumerate}
\end{theo}

When $M$ has positive constant sectional curvature one can apply Theorem \ref{theo2} but in this case we can apply Remark \ref{remark} to get a  better result.

\begin{theo}\label{theo1}
	Let $(M,\prs_{TM})$ be a  Riemannian manifold with positive constant sectional curvature $c$. Then, for $k$  close to $\frac{2}{c}$,  $(E^{(r)}(M,k),h)$ has nonnegative sectional curvature.		
\end{theo}
\begin{proof}
	 Suppose that $M$ of constant curvature $c$. Let us find in this case a $\mathbf{K}$ as in \eqref{est}. For any $X,Y,Z\in\Ga(TM)$ and $F\in\Ga(\mathrm{so}(TM))$, we have
	 \[ |R^{\na^E}(X,Y)(Z+F)|\leq |R^{\na^E}(X,Y)Z|+|R^{\na^E}(X,Y)F|. \]
	 From Proposition \ref{pr3}, we get that
	 \[ |R^{\na^E}(X,Y)Z|\leq4|\varpi||X||Y||Z| \esp R^{\na^E}(X,Y)F=2\varpi\left(F(X)\wedge Y+X\wedge F(Y)  \right). \]Let us compute $|F(X)\wedge Y|$. Let $(X_i)_{i=1}^n$ be a local orthonormal frame of $TM$. Then
	 \begin{eqnarray*}
	 |F(X)\wedge Y|^2&=&-k\tr((F(X)\wedge Y)^2)\\
	 &=&k\sum_{i=1}^n\langle F(X)\wedge Y(X_i),F(X)\wedge Y(X_i)\rangle_{TM}\\
	 &=&k\sum_{i=1}^n\langle \langle Y,X_i\rangle_{TM}F(X)-\langle F(X),X_i\rangle_{TM}Y,\langle \langle Y,X_i\rangle_{TM}F(X)-\langle F(X),X_i\rangle_{TM}Y\rangle_{TM}\\
	 &=&2k|F(X)|^2|Y|^2+2k\langle F(X),Y\rangle_{TM}^2\leq 4|F|^2|X|^2|Y|^2.
	 \end{eqnarray*}Finally,
	 \[ |R^{\na^E}(X,Y)(Z+F)|\leq8|\varpi||X||Y|(|Z|+|F|). \]So we can take $\mathbf{K}=8|\varpi|$ which goes to zero when $k$ goes to $\frac{2}{c}$. Thus when $k$ is close to $\frac{2}{c}$ the inequality \eqref{eqcurv1} holds and we get the desired result.
		\end{proof}
		
\subsection{Riemannian manifolds whose $(E^{(r)}(M,k),h)$ is Einstein}

It has been proved in \cite{book} that $(T^{(r)}M,h)$ is Einstein if and only if $\dim M=2$ and either $M$ is flat or has constant curvature $\frac1{r^2}$. We have a more rich situation  in the case of $(E^{(r)}(M,k),h)$.

\begin{theo} Let $(M,\prsm)$ be a connected Riemannian manifold. Then:\begin{enumerate}
	
	\item $(E^{(r)}(M,k),h)$ is Einstein with Einstein constant $\la$ if and only if the Riemannian covering of $(M,\prsm)$ is locally isometric to the Riemannian product ${\mathbb{S}}^p\left(\sqrt{\frac{k}2}\right)\times\ldots\times {\mathbb{S}}^p\left(\sqrt{\frac{k}2}\right)$ of $q$ spheres of dimension $p$ and radius $\sqrt{\frac{k}2}$  with
	\[ \la=\frac{2(p-1)}k=\frac{qp(qp+1)-4}{2r^2}. \]
	\item $(E^{(r)}(M,k),h)$ can never have a constant sectional curvature.\end{enumerate}
	
\end{theo}
\begin{proof} This is an immediate consequence of Theorems \ref{einstein} and \ref{supra}.
	\end{proof}

\subsection{Scalar curvature of $(E^{(r)}(M,k),h)$}		
As an application of Theorem \ref{scalar} , we have the following result:

\begin{theo}
	 Suppose that $(M,\prsm)$ has constant sectional curvature $c$. Then
	 $(E^{(r)}(M,k),h)$ has constant scalar curvature if and only if either $n=3$, $c=0$ or $c>0$ and $k=\frac2c$.
	 \end{theo}
	 
	 \begin{proof} The scalar curvature $\tau$ is giving by, for $(x,Z+F) \in E^{(r)}(M,k)$ \[ \tau(x,Z+F) = n(n-1)c+\frac{1}{r^2}(m-1)(m-2)-\frac{1}{4}\xi_x(Z+F,Z+F), \]
	 	where
	 	\[ \xi_x(Z+F,Z+F)=2\varpi^{2}(n-1)|Z+F|^2+2\varpi^{2}(n-3)|F|^2,\quad \varpi=\frac14c(2-ck). \]So we get the desired result.
	 	\end{proof}

We end this subsection by giving  all two-dimensional Riemannian manifolds $(M,\prsm)$ for which $(E^{(r)}(M,k),h)$  has constant scalar curvature. 
\begin{pr} Let $(M,\prsm)$ be a 2-dimensional Riemannian manifold with curvature $R^M(X,Y)=-CX\wedge Y$ with $C \in C^\infty(M)$. Then,
	for any $X,Y\in\Ga(TM)$ and $F\in\Ga(\mathrm{so}(TM))$,\[ R^{\na^E}(X,Y)Z=-\varpi X\wedge Y(Z)+\frac{1}{2}Z(C)X\wedge Y\esp
	R^{\na^E}(X,Y)F=-\varpi [X\wedge Y,F]+ k\langle F(X),Y \rangle_{TM} \mathrm{grad}(C) ,\]
	where $\varpi=\frac12C(2-kC)$ and  $X\wedge Y$ is the skew-symmetric endomorphism of $TM$ given by
	\[ X\wedge Y(Z)=\langle Y,Z\rangle_{TM} X-\langle X,Z\rangle_{TM} Y. \]
\end{pr}	
\begin{proof}	
	According to \eqref{r},
	\[ R^{\na^E}(X,Y,Z)=R^{M}(X,Y,Z)+H_{Y}H_{X}Z-H_{X}H_{Y}Z-\frac12\na^M_Z(K^M)(X\wedge Y), \] where
	 $H_XY=-\frac12R^M(X,Y)=\frac12CX\wedge Y$ and
	\[  \langle H_XF,Y\rangle_{TM}=-\frac12k\;\tr( F\circ R^M(X,Y))=-\frac12Ck\sum_{i=1}^n\langle F(X_i),X\wedge Y(X_i)\rangle_{TM} =-Ck\langle F(Y),X\rangle_{TM}.\]
	Thus $H_XF=CkF(X)$ and
	\[ H_{Y}H_{X}Z-H_{X}H_{Y}Z=\frac12C^2k(X\wedge Z(Y)-Y\wedge Z(X))=\frac12C^2kX\wedge Y(Z). \]
	Moreover,
	\begin{eqnarray*}
	\na^M_Z(K^M)(X\wedge Y)&=&\na^M_Z(K^M(X\wedge Y))-K^M(\na^M_ZX\wedge Y)-K^M(X\wedge \na^M_ZY)\\
	&=&-\na^M_Z(CX\wedge Y)+C\na^M_ZX\wedge Y+CX\wedge \na^M_ZY\\
	&=&-Z(C)X\wedge Y.
	\end{eqnarray*}By adding the expressions above we get the first formula.
	
	On the other hand,
	$$R^{\na^E}(X,Y)F=\left\{(R^{\na^E}(X,Y)F)_{TM} \right\}
	+\left\{  [R^{M}(X,Y),F]+H_{Y}H_{X}F-H_{X}H_{Y}F\right\},$$where
	\begin{eqnarray*}
	\langle (R^{\na^E}(X,Y)F)_{TM},Z\rangle_{k}&=&-\langle R^{\na^E}(X,Y)Z,F\rangle_k\\
	&=&-\frac{1}{2}Z(C)\langle X\wedge Y,F\rangle_k\\
	&=&-\frac{k}{2}\langle\mathrm{grad}(C),Z\rangle_{TM}\sum_{i=1}^n\langle X\wedge Y(X_i),F(X_i)\rangle_{TM}\\
	&=&k\langle F(X) , Y \rangle_{TM}\langle\mathrm{grad}(C),Z\rangle_{TM}.
	\end{eqnarray*}Thus $(R^{\na^E}(X,Y)F)_{TM}=k\langle F(X) , Y \rangle_{TM}\mathrm{grad}(C).$ Furthermore,
	\[ 	[H_Y,H_X]F=Ck(H_YF(X)-H_XF(Y))=-\frac12Ck(R^M(Y,F(X))+R^M(F(Y),X))=-\frac12C^2k([F,X\wedge Y]).\]
	This completes the proof.
\end{proof}
\begin{theo} Let $(M,\prsm)$ be a 2-dimensional Riemannian manifold. Then
	$(E^{(r)}(M,k),h)$ has constant scalar curvature if and only if $(M,\prsm)$ has constant curvature  $C=0$ or $C=\frac{2}{k}$.
\end{theo}
\begin{proof}
	We choose an orthonormal basis $(X_1,X_2)$ such that $Ric^M(X_i)=\rho_i X_i$ and we put $F_{12}=\frac{1}{\sqrt{2k}}X_1\wedge X_2$. The family $(X_1,X_2,F_{12})$ is a local orthonormal frame of $E(M,k)$. We have, for any vector field $Z$,\[ R^{\na^E}(X_1,X_2)Z=-\frac12C(2-kC)X_1\wedge X_2(Z) \esp R^{\na^E}(X_1,X_2)F_{12}=-\sqrt{\frac{k}{2}} \mathrm{grad}(C) .\]
	Then,
	\begin{eqnarray*}
		\xi(X_i,X_i)&=& 2\varpi^2+k(X_i(C))^2,\; i=1,2\\
		\xi(F_{12},F_{12}) &=& k|\mathrm{grad}(C)|^2,\\
		\xi(X_1,X_2)&=& kX_1(C)X_2(C),\\
		\xi(X_i,F_{12}) &=& \varpi^2 \sqrt{2k} \langle \mathrm{grad}(C),X_1\wedge X_2(X_i) \rangle,\; i=1,2.
	\end{eqnarray*}
	On the other hand
	\[ |R^{\na^E}|^2 = 4\varpi^2+2k|\mathrm{grad}(C)|^2.\]
	Suppose that $(E^{(r)}(M,k),h)$ has constant scalar curvature. The equation (\ref{s1}) gives, for $F_{12}$ \[ 4\varpi^2-k|\mathrm{grad}(C)|^2=0. \]
	We eliminate $|\mathrm{grad}(C)|^2$ in the equation (\ref{s2}), to find \[ 24 C -3 C^2(2-kC)^2=\mathrm{constant}. \]
	So $C$ must be constant and $C=0$ or $C=\frac{2}{k}$.
\end{proof}

\section{The Sasaki metric with positive scalar curvature on the unit bundle of three dimensional unimodular Lie groups}\label{section5}

The purpose of this section is to prove the following result.
\begin{theo}\label{unimodular}
	Let $G$  be a three dimensional connected unimodular Lie group. Then there exists a left invariant Riemannian metric on $G$ such that $(T^{(1)}G,h)$ has positive scalar curvature.
\end{theo}

\begin{proof}
	 Let $G$ be a connected $3$-dimensional unimodular Lie group with left invariant metric. By using an argument developed in \cite{milnor}, 
	 there exists an orthonormal basis $(X_1,X_2,X_3)$ of left invariant vector fields such that  \[ [X_1,X_2]=mX_3,\hspace*{0.2cm} \;  [X_1,X_3]=nX_2 \esp [X_2,X_3]=pX_1 .\]
	
	By straightforward computation using the Koszul formula, we get that the Levi-Civita connexion in this case is given by 
	\begin{eqnarray*}
		\na_{X_1}&=&(m+n-p)X_2 \wedge X_3,\\
		\na_{X_2}&=&(m+n+p)X_1 \wedge X_3,\\
		\na_{X_3}&=&(-m+n+p)X_1\wedge X_2.
	\end{eqnarray*}
	Thus, we obtain the following formula for the Riemann curvature tensor $R^{G}$ \[R^{G}(X_i,X_j)=\mu_{ij}X_i\wedge X_j,\]
	where $i,j \in \{1,2,3\}$, $i<j$  and $\mu_{ij}$ are constants given by
	\begin{eqnarray*}
		\mu_{12} &=& \frac{1}{4} \left( ( p+n+m)(-n-p+m)+ (-p+m+n)(-n-p+m) +(-p+n+m)(p+n+m) \right),\\
		\mu_{13} &=& -\frac{1}{4} \left( (-p+n+m)(-n-p+m)+ (p+n+m)(-n-p+m)-(-p+m+n)(p+n+m) \right),\\
		\mu_{23} &=& -\frac{1}{4} \left( (-p+m+n)(p+n+m)- (-p+m+n)(-n-p+m) +(p+n+m)(-n-p+m) \right).
	\end{eqnarray*}
	The scalar  curvature of the unit tangent sphere bundle $(T^{(1)}G,h)$ of $G$ equipped with the Sasaki metric is given by, for any $(x,a)\in T^{(1)}G$ \[ \tau(x,a) = 1-\mu_{12}-\mu_{13}-\mu_{23}-\frac{1}{4}\xi_x(a,a), \]
	where $ \xi(a,a)=\sum_{i,j=1}^3| R^G(X_i,X_j)a|^2$.  We have \[ \xi(X_1,X_1) =2( \mu_{12}^2+\mu_{13}^2 ),\hspace*{0.2cm}\;  \xi(X_2,X_2) =2( \mu_{12}^2+\mu_{23}^2 ) \esp  \xi(X_3,X_3) =2( \mu_{13}^2+\mu_{23}^2 ) .\] Put 
	\begin{eqnarray*}
		\la_1 &=& \mu_{12}^2+\mu_{13}^2 +4(\mu_{12}+\mu_{13}+\mu_{23}-1),\\
		\la_2 &=& \mu_{12}^2+\mu_{23}^2 +4(\mu_{12}+\mu_{13}+\mu_{23}-1),\\
		\la_3 &=& \mu_{13}^2+\mu_{23}^2 +4(\mu_{12}+\mu_{13}+\mu_{23}-1).
	\end{eqnarray*}
	Then, the scalar curvature $\tau$ of $(T^{(1)}G,h)$ is positive if and only if  $\la_i< 0$ for all $i \in \{1,2,3\}$.
	There are values for parameters $m,n$ and $p$ for which $\la_i$ is negative for all  $i \in  \{1,2,3\}.$ 
	\begin{enumerate}
		\item For $m=\frac{1}{2}$, $n=\frac{1}{3}$  and $p=\frac{1}{4}:$ In this case, the Lie group $G$ is isomorphic to the group $SO(3)$, or $SU(3)$,
		\[ \la_1=-\frac{543127}{165888},\hspace*{0.2cm}\; \la_2=-\frac{545675}{165888} \esp \la_3=-\frac{542035}{165888}.\]
		\item For $m=\frac{1}{2}$, $n=\frac{1}{3}$ and $p=-\frac{1}{4}:$ $G \cong  SL(2,\R)$ or $O(1,2)$,
		\[ \la_1=-\frac{505879}{165888},\hspace*{0.2cm}\; \la_2=-\frac{504059}{165888} \esp \la_3=-\frac{522259}{165888}.\]
		\item For $m=\frac{1}{2}$, $n=\frac{1}{3}$ and $p=0:$  $G \cong E(2)$,
		\[ \la_1=-\frac{33547}{10368},\hspace*{0.2cm}\; \la_2=-\frac{33347}{10368} \esp \la_3=-\frac{33847}{10368}.\]
		\item For $m=\frac{1}{2}$, $n=\frac{1}{3}$ and $p=0:$ $G \cong  E(1,1),$ \[ \la_1=-\frac{33547}{10368},\hspace*{0.2cm}\; \la_2=-\frac{33347}{10368} \esp \la_3=-\frac{33847}{10368}.\]
		\item For $m=\frac{1}{2}$, $n=-\frac{1}{3}$ and $p=0:$ $G \cong  H(3,\R), $ \[ \la_1=-\frac{33547}{10368},\hspace*{0.2cm}\; \la_2=-\frac{33347}{10368} \esp \la_3=-\frac{33847}{10368}.\]
	\end{enumerate}
	
\end{proof}

\end{document}